\documentclass[11pt]{article}
 \usepackage{amsmath, amsthm, amsfonts, amssymb}
\usepackage{color}
\usepackage{url}

\hoffset= - 0.5 cm \voffset= - 0.1 cm \newtheorem{theorem}{Theorem}[section]
\newtheorem{lemma}[theorem]{Lemma}
\newtheorem{proposition}[theorem]{Proposition} 
\newtheorem{corollary}[theorem]{Corollary}
\newtheorem{remark}[theorem]{Remark}
\newtheorem{definition}[theorem]{Definition}
\newtheorem{example}[theorem]{Examples}     
\newtheorem{hypothesis}{Hypothesis}

\def\qed{\hfill \hbox{\hskip 6pt\vrule
width6pt height6pt depth1pt  \hskip1pt}
\smallskip}

\newcommand{\bth}{\begin{theorem}}
\renewcommand{\eth}{\end{theorem}}
\newcommand{\bpr}{\begin{proposition}}
\newcommand{\epr}{\end{proposition}}
\newcommand{\bco}{\begin{corollary}}
\newcommand{\eco}{\end{corollary}}
\newcommand{\ble}{\begin{lemma}}
\newcommand{\ele}{\end{lemma}}
\newcommand{\bpf}{\begin{proof}}
\newcommand{\epf}{\end{proof}}
\newcommand{\bex}{\begin{example}}
\newcommand{\eex}{\end{example}}
\newcommand{\bdf}{\begin{definition}}
\newcommand{\edf}{\end{definition}}
\newcommand{\bre}{\begin{remark}}
\newcommand{\ere}{\end{remark}}

 \newcommand{\bal}{\begin{aligned}}
\newcommand{\eal}{\end{aligned}}
\newcommand{\ben}{\begin{enumerate}}
\newcommand{\een}{\end{enumerate}}

\def\P{{\mathbb P}}
\def\R{{\mathbb R}}
\def\C{{\mathbb C}}
\def\E{{\mathbb E }}

\def\hh{{\vskip 1mm}}

\begin{document}
 \title {   On Davie's  {\!}      uniqueness  {\!}    
for {\!}      some  degenerate SDEs  {\!}       }

\author{\normalsize          Enrico {\!} Priola {\!} \\ \small Dipartimento {\!} {\!} di {\!}   
 Matematica,  
Universit\`a {\!}  di {\!}  Pavia\\
 \small Via {\!} Adolfo {\!} Ferrata {\!} 5\\ \small 27100 Pavia {\!} \ \  enrico.priola@unipv.it
 }
\date{}

 \maketitle
 
\abstract {  We consider {\!}     singular {\!}        SDEs  {\!}    like 
 \begin{equation} \label{ss}
dX_t = b(t, X_t) dt + A X_t dt + \sigma(t) d{L}_t , \;\; t \in     [0,T], \;\; X_0 =x \in \R^n,
\end{equation}
 where $A$ is  {\!}    a {\!} real $n \times  n $ matrix, i.e.,  $A \in {\R}^n \otimes    {\R}^n$,  $b$ is  {\!}    bounded and H\"older {\!}     continuous,  
 $\sigma  : [0,\infty) \to {\!} {\R}^n \otimes    {\R}^d $ is  {\!}       a {\!}    locally bounded function  and $L= ({L}_t)$ is  {\!}    an $\R^d$-valued L\'evy process, $1 \le d \le n$. We show  that 
    strong existence and uniqueness  {\!}     together {\!}     with  $L^p$-Lipschitz dependence    on the initial condition  $x $ imply  Davie's  {\!}     uniqueness  {\!}    or {\!}     path by path uniqueness.  
  This  {\!}     extends  {\!}      a {\!} result   
 of [E. Priola,   AIHP, 2018] proved for {\!}     \eqref{ss}   
  when    $n=d$, $A=0$ and $\sigma(t) \equiv I  $.  
  We apply the result to {\!} some singular {\!}     degenerate SDEs  {\!}    associated to {\!} the  kinetic transport operator {\!}     $ \frac{1}{2} \triangle_v f + $ ${v \cdot \partial_{x}f} $   $+F(x,v)\cdot \partial_{v}f $  when  
      $n =2d $ and 
        $L$ is  {\!}    an ${\R}^d$-valued  Wiener {\!}     process. For {\!}      such equations  {\!}    strong existence and uniqueness  {\!}    are known under {\!}     H\"older {\!}     type   conditions  {\!}    on $b$.  We show that  in addition also {\!}  Davie's  {\!}     uniqueness  {\!}     holds.    
 }
 
 \vspace{2.7 mm}


\vspace{2.7 mm}

 \noindent {\bf Mathematics  {\!}     Subject Classification (2010):} 
60H10, 60J75, 34F05.

 \section{Introduction  }
 
 Davie's  {\!}     type uniqueness  {\!}     or {\!}     path-by path uniqueness  {\!}    for {\!}     stochastic differential equations  {\!}    (SDEs) has  {\!}    recently received a {\!} lot of attention (cf. \cite{D}, \cite{Fl}, \cite{flandoli}, \cite{Sh}, \cite{gubi}, \cite{Pr}, \cite{wresch}, \cite{mytnik} and see the references  {\!}    therein). 
  
  This  {\!}    type of uniqueness  {\!}    has  {\!}    been introduced in \cite{D} where A.M. Davie  considered a {\!}   SDE like $dX_t = b(t, X_t)\,dt + dW_t$, $X_0=x \in {\R}^n$, driven by an ${\R}^n$-valued  Wiener {\!}     process  {\!}     $W$ and having a {\!} bounded and measurable drift coefficient  $b$. For {\!}     such equations  {\!}     pathwise (or {\!}     strong) uniqueness  {\!}    in the sense of K. It\^o
   had already been established  by A.J. Veretennikov in \cite{Ver} even  with a {\!} multiplicative noise.   The paper {\!}     \cite{D} improves  {\!}    \cite{Ver}   by showing that  the previous  {\!}     equation has  {\!}    a {\!} unique solution  for {\!}     almost all choices  {\!}    of the driving Brownian path.   In other {\!}     words, 
 adding  a {\!} single Brownian  path    regularizes  {\!}    a {\!}  singular {\!}      ODE (cf. \cite{Fl} and \cite{gubi}).  
    
    Here we study Davie's  {\!}     type uniqueness  {\!}      for {\!}     singular {\!}      SDEs  {\!}    like
  \begin{gather} \label{25}
dZ_t = b(t, Z_t) dt + A Z_t dt + \sigma(t) d{L}_t , \;\; t \in [s,T], \;\; Z_s  {\!}    =x \in {\R}^n,
\end{gather}
 $T>0$, $s  {\!}    \in [0,T)$.   Here $A \in {\R}^n \otimes    {\R}^n$, 
   $\sigma  : [0,\infty) \to {\!} {\R}^n \otimes    {\R}^d$   is  {\!}    a {\!}  Borel and locally bounded  function, $1 \le d \le n$,
and ${L} = ({L}_t)$ is  {\!}    a {\!} $d$-dimensional L\'evy process  {\!}    defined on a {\!} complete probability space $(\Omega,$ $ {{\mathcal F}}, \P)$; ${\R}^{n} \otimes    {\R}^d$ indicates  {\!}    the space of all $n \times  d$ real matrices.
 
 The drift coefficients  {\!}     
$b: [0,T] \times  {\mathbb R}^n \to {\!} {\mathbb R}^n $ is  {\!}    Borel measurable, bounded and
 ${\beta}$-H\"older {\!}     continuous  {\!}    in the $x$-variable, uniformly in $t$, i.e.,    $ b \in L^{\infty}(0, T; C_{b}^{0,{\beta}}({\R}^n; {\R}^n))$. We concentrate on    the singular {\!}     case 
 ${\beta} \in (0,1)$  but   the results  {\!}    can be extended to {\!} the Lipschitz case ${\beta} = 1$.  
 
 We  generalize   a {\!} theorem
 proved in \cite{Pr} for {\!}      equations  {\!}    \eqref{25}  when $n=d$, $A=0$ and $\sigma(t) \equiv  I$.             In \cite{Pr}  it is  {\!}      shown in particular {\!}        that if  strong existence and uniqueness  {\!}    hold for {\!}     the SDE   and further {\!}      there is  {\!}    Lipschitz dependence in $L^p$-norm on  the initial condition $x$ (cf. Hypothesis  {\!}    \ref{qq} below) then we have   Davie's  {\!}      uniqueness  {\!}    for {\!}     the SDE (cf. Theorem \ref{main1}).
 
 Setting $M_t = \int_0^t \sigma(s) d{L}_s,$ 
 equation  \eqref{25}  can be   written as  {\!}         
  \begin{gather} \label{SDE}
 Z_{t}(\omega) = x + \int_{s}^{t} [ b (v, Z_{v}(\omega)) + AZ_{v}(\omega)]  dv \, + \, M_{t}(\omega) - M_s(\omega),   
 \end{gather}   
$\omega \in \Omega$ (we are considering the  stochastic integral $M_t$ as  {\!}    in Section 4.3 of \cite{A}). Note 
    that   $M = (M_t)$   is  {\!}    an example of  additive process  {\!}    with values  {\!}    in ${\R}^n$ (see, for {\!}     instance, Chapter {\!}     2 in \cite{RS}).  Additive processes  {\!}    generalize L\'evy processes  {\!}    by relaxing
the stationarity condition on the increments
(cf. \cite{Sato}, \cite{Ito}, \cite{Sato1} and the references  {\!}    therein). Note that in \cite{Pr} one considers  {\!}      $M ={L}$.  

Since in general $M$ does  {\!}    not have stationary increments,
 in order {\!}     to {\!} prove the uniqueness  {\!}    result we have to {\!} show that  the  proofs  {\!}    in \cite{Pr} can be carried out  without using the stationarity of  increments  {\!}    of  the driving process.   On the other {\!}     hand, 
    \eqref{SDE} is  {\!}    not covered by   \cite{Pr} even if $\sigma(t) $ is  {\!}    a {\!} constant matrix. Indeed  
the coefficient $Ax $ is  {\!}    not bounded and   in general one  cannot truncate such term and localize as  {\!}    in the end of   Section 5  of \cite{Pr}.  Truncating  $x \mapsto {\!} Ax$, when $d<n$, can make difficult to {\!}  obtain  strong uniqueness  {\!}    and Lipschitz dependence on $x$ (cf. \cite{CdR}, \cite{WaZa2}, \cite{men} and see Remark \ref{forse}).   

Before stating our {\!}     theorem on Davie's  {\!}       uniqueness  {\!}      we make assumptions  {\!}    on  the terms  {\!}    appearing in \eqref{25}:  $b(t,x)$, $A$, $\sigma(t)$ and the $d$-dimensional  L\'evy process  {\!}    $L$   defined on a {\!} complete probability space $(\Omega, {\cal F}, \P)$.   Recall that the law of $L$ is  {\!}    characterized by the L\'evy-Khintchine formula {\!} \eqref{levii}.  
\begin{hypothesis} \label{qq}     
  {\em 
(i) For {\!}     any $s  {\!}    \in [0,T]$  and $x \in {{\R}}^n$
 on $(\Omega, {\mathcal F}, \P)$ 
there exists  {\!}    a {\!} strong solution $(Z_{t}^{s,x})_{t \in [0,T]}$  to {\!} \eqref{SDE}.

\hh  (ii) Let us  {\!}    fix $s  {\!}    \in [0,T]$.
Given  any two {\!} 
strong
solutions  {\!}    $(Z_t^{s,x})_{t \in [0,T]}$
 and $(Z_t^{s,y})_{t \in [0,T]}$
 defined on  $(\Omega, {\mathcal F}, \P)$ which both solve   \eqref{SDE} 
with respect to {\!}  $A$, $\sigma(t)$,  ${L}$ and 
 $b$
(starting \ at $x$ and $y \in {\R}^n$, respectively, at time $s$)
we have, 
for {\!}     any $p \ge  2$,  
 \begin{equation} \label{ciao221}
\sup_{s {\,}    \in [0,T]} \E \big [ \displaystyle{ \sup_{s {\,}    \le t \le T}} \, |\, 
Z_t^{s,x} \; - \, Z_t^{s,y} |^p \big] 
\le C_T \;
 |x-y|^p, \quad  x,\, y \in {{\R}}^n,
\end{equation}
 with $C_T \!= C \big( \text{Law}({L})$, $A$, $\sigma$,  $ b$ ,
 ${\beta},$ $  n,  p, T \big) >0$      
   independent of $s$, $x,$ $y$. \qed
}   
\end{hypothesis}  
  \begin{theorem} \label{main1} Let us  {\!}    consider {\!}     \eqref{SDE} 
   with $A \in {\R}^n \otimes    {\R}^n$,   $b \in   L^{\infty}(0, T; C_{b}^{0,{\beta}}({\R}^n; {\R}^n))$, ${\beta} \in (0,1)$, and $\sigma  : [0, \infty) \to {\!}  {\R}^n \otimes      {\R}^d$ locally bounded.
  Assume 
  Hypothesis  {\!}    \ref{qq} and    
suppose $\E [ |{L}_1|^{\theta} ]$ $< \infty $, for {\!}     some $\theta  \in (0,1)$.  

Setting $\tilde b(t,x) = b(t,x) + A x$,   
 there exists  {\!}    an  event $\Omega' \in {\mathcal F}$ with $\P(\Omega') =1$ such that for {\!}     any $\omega \in \Omega'$, $x \in {{\R}}^n$, the integral equation
\begin{equation} \label{davi}
 g(t) = x + \int_{0}^{t} \tilde b(v, g(v) + M_{v}(\omega))  dv,\quad  t \in [0,T], 
\end{equation} 
has  {\!}    {exactly} one solution $g$ \ in $C([0,T];$ ${{\R}}^n)$. 
\end{theorem}
 \noindent The previous  {\!}    result will be deduced from   Theorem \ref{h32} which extends  {\!}      Theorem 5.1 in \cite{Pr}. We remark that   in Corollary \ref{unb}  we  will show  Davie's  {\!}    uniqueness  {\!}    for {\!}     SDE \eqref{SDE} when $b$  is  {\!}    locally H\"older {\!}      continuous  {\!}       by a {\!} standard localization procedure. 
   
A special case of \eqref{SDE}  is  {\!}    the following SDE     
 \begin{equation} \label{d}
dZ_t = A Z_t dt + b(t, Z_t)dt + C d L_t ,\;\; Z_s  {\!}    =x \in \R^n.
\end{equation}  
 Here
 $A \in {\R}^n \otimes    {\R}^n$ and $C \in {\R}^n \otimes    {\R}^d$ are   given matrices. 
  When $L=W$ is  {\!}    a {\!} $d$-dimensional Wiener {\!}     process, $d<n$,  pathwise uniqueness, flow   and differentiability properties  {\!}    of the solutions  {\!}    to {\!} \eqref{d} have  been  recently   investigated also {\!} under {\!}     H\"older {\!}     type conditions  {\!}     on $b$ (see, for {\!}     instance,  \cite{CdR}, \cite{WaZa2}, \cite{FFPV}, \cite{men}, \cite{Z}   
   which consider {\!}      more general degenerate SDEs  {\!}    as  {\!}    well).  
    
As  {\!}    an  example of degenerate SDE of the form \eqref{d} we consider {\!}      
\begin{equation}
\label{eq-SDE}
\begin{cases}
d X_t &  =V_t d t,\; \;  d V_t = F \left(  X_t \right)  d t + d W_{t} \\
X_0   &  =x_{0} \in \R^d, \;\; \quad V_0  =v_{0} \in \R^d,
\end{cases}   
\end{equation}  
 (see Section 5 for {\!}     the case  in which   $F$  possibly depends  {\!}    also {\!} on $V_t$; see also {\!} Remark \ref{forse} for {\!}     more general SDEs).  Equation \eqref{eq-SDE}  involves  {\!}     the velocity-position  
of a {\!} particle that moves  {\!}     according to {\!} the Newton second law  
in a {\!} force-field   $F$ and under {\!}     the  action of noise (see \cite{Pav} and the references  {\!}    therein).   
It 
 is  {\!}    associated to {\!} the well-studied  kinetic transport operator {\!}      $ \frac{1}{2} \triangle_v f + $ ${v \cdot \partial_{x}f} $   $+F(x)\cdot \partial_{v}f $. 
 An application of \eqref{eq-SDE} to {\!} the study of   
  singular {\!}      kinetic transport SPDEs  {\!}     is  {\!}    given in  \cite{FFPV}.

 In this  {\!}    case $n = 2d$, $d \ge 1$, and  $W$ is  {\!}    a {\!} $d$-dimensional Wiener {\!}     process. Moreover {\!}     
 $b(z) = b(x,v) = \begin{pmatrix}    0  \\ F (x)   \end{pmatrix}:$ $ \R^{2d} \to {\!} \R^{2d}$. When   $F : \R^d \to {\!} \R^d$ has  {\!}    at most a {\!} linear {\!}     growth  and it is  {\!}    locally $\beta$-H\"older {\!}     continuous  {\!}     with $\beta \in (2/3, 1)$  it is  {\!}    known that there exists  {\!}    a {\!} unique  strong solution (the value   $2/3$  is  {\!}    the critical H\"older {\!}     index for {\!}     strong uniqueness, cf. \cite{CdR}, \cite{WaZa2} and \cite{men}). Under {\!}     these assumptions  {\!}    applying  Corollary \ref{unb}
  we can show that  also {\!}   Davie's  {\!}    type uniqueness  {\!}    holds  {\!}    for {\!}     \eqref{eq-SDE}.

 We   mention \cite{gubi} where  in particular {\!}      path-by-path uniqueness  {\!}    for {\!}     
 SDEs  {\!}    with additive fractional Brownian noise is  {\!}    investigated. 
 Finally,  remark that     path-by-path uniqueness  {\!}    has  {\!}    been also {\!} studied in infinite dimensions  {\!}    for {\!}     some  
 SPDEs. We refer {\!}     to {\!}  \cite{wresch} and \cite{mytnik}.

 \section{Notation and preliminary results}
  
 The Euclidean norm in ${\R}^k $, $k \ge 1,$ and the inner {\!}     product are indicated by $|\cdot|$ and $\langle \cdot , \cdot  \rangle $ respectively.  Moreover, ${\cal B}(A)$ indicates  {\!}    the Borel $\sigma$-algebra {\!} of a {\!} Borel set $A \subset \R^k$.      
 
    We denote by 
  $C_{b}^{0,{\beta}}({\R}
 ^{n};{\R}^{k})$, ${\beta} \in (0,1)$,
   the space of all { ${\beta}$-H\"older {\!}     continuous}
   functions  {\!}    $f$, i.e., $f$ verifies
$$
 \begin{array}{l}
 [f]_{C^{0,{\beta}}_b} = [ f]_{{\beta}}:=\sup_{x\neq x'\in\mathbb{R}^{n}}
 {(|f(x)-f(x')|}\, {|x-x'|^{-{\beta}}})<\infty
\end{array}
$$
  Note that 
 $C_{b}^{0,{\beta}}({\R}) 
 ^{n};\mathbb{R}^{k})$ is  {\!}    a {\!} Banach space with the norm $
   \| \cdot \|_{{\beta}}$ $ = \| \cdot \|_0 + [\cdot ]_{{\beta}}$. 
 
To {\!} study  \eqref{25} 
we require that  $b$   belongs  {\!}    to {\!} $ L^{\infty}(0, T; C_{b}^{0,{\beta}}({\R}^n; {\R}^n))$. Hence 
 $b:$ $ [0,T] \times  {\R}^n \to {\!} {\R}^n $ is  {\!}    Borel  and bounded, $b(t, \cdot ) \in  C_{b}^{0,{\beta}}({\R}^n; {\R}^n)$, $t \in [0,T]$, and
   $ [b]_{{\beta}, T} = \sup_{t \in [0,T]} [b(t, \cdot)]_{C^{0, {\beta}}_b} < \infty.$ We also
set $\| b\|_{{\beta},T} = [ b]_{{\beta}, T}$ $+ \| b\|_{0,T}$; \  
$\|b\|_{0, T} $ $= \sup_{t \in [0,T], x \in {\R}^d} |b(t,x)|$,
  ${\beta} \in (0,1)$.   Finally, a {\!} function   $g$ $ \in C_{0}^{\infty}(\R^n)$ if $g$ belongs  {\!}    to
 $ C^{\infty}(\R^n)  $ and  has  {\!}    compact support.

\medskip     
Let ${L} =({L}_t)$  be a {\!} L\'evy process  {\!}    with values  {\!}    in ${\R}^d$ defined on a {\!} complete  probability space $(\Omega, {\mathcal F}, \P)$ (see \cite{Sato}, \cite{Kr} and \cite{A}).  Thus  {\!}    $L$
 has  {\!}    independent and  stationary  increments,      c\`adl\`ag  trajectories  {\!}    and    $L_0=0$, $\P$-a.s..     We will denote by ${L}_{s-}(\omega)$ the left-limit in $s>0$, $\omega \in \Omega$.   
  
 For {\!}     $0 \le s  {\!}    < t < \infty$ we denote by  ${\mathcal F}_{s,t}^{L}$
   the completion of the $\sigma$-algebra {\!} generated by  ${L}_r {\!}     - {L}_s$, $ s  {\!}    \le r {\!}     \le t$. We also {\!} define ${\mathcal F}_{0,t}^{L} = {\mathcal F}^{L}_t. $ Since ${L}$ has  {\!}    independent increments  {\!}    we have that ${L}_q - {L}_p$ is  {\!}    independent of ${\mathcal F}_{p}^{L}$ when  $0 \le  p< q$.

 We say that 
$\tilde \Omega \subset \Omega$ is  {\!}     an   almost sure event if $\tilde \Omega \in {\mathcal F}$ and $\P(\tilde \Omega)=1$. As  {\!}    in \cite{Pr}
 we write ${\tilde \Omega}_\mu$
to {\!} stress  {\!}    that   $\tilde \Omega$ possibly depends  {\!}    also {\!} on the  parameter {\!}     $\mu$  ($\tilde \Omega_{\mu}$ may change from one proposition to {\!} another). For {\!}     instance, we write $\tilde \Omega_{s,x}$   
 or {\!}     $\Omega_{s,x}'$.

\smallskip  Recall the exponent $\phi$ of ${L}$. This  {\!}    is  {\!}     
   a {\!} function $\phi$ $ : {\R}^d  \to {\!} {\C}$ such that 
 $
 \E [e^{i   \langle  {L}_t, k\rangle}] = e^{- t
\phi(k)}, $ $  k \in {\R}^d,$ $ t \ge 0.
$
  The L\'evy-Khintchine formula {\!}  says  {\!}    that 
 \begin{gather} \label{levii}
 \phi(k)= \frac{1}{2}\langle Q k,k  \rangle   - i {\!}  \langle a, k\rangle
- \int_{{\R}^d} \! \! \big(  e^{i   \langle k,y \rangle }  - 1 -  { i {\!}  \langle k,y
\rangle} \, {1}_{\{ |y| \le 1\}} \, (y) \big ) \nu (dy), 
\end{gather}
 $k \in {\R}^d,$ where  $Q$ is  {\!}    a {\!} symmetric non-negative definite $d \times  d$-matrix, $a
\in {\R}^d$ and $\nu$ is  {\!}    a {\!} $\sigma$-finite (Borel)
measure on ${\cal B}({\R}^d)$, such that  
 $ \int_{{ {\R}^d}}    (1 \wedge $ $|v|^2 ) \, \nu(dv)$
$ <\infty,$  $\nu (\{ 0\})=0$;
$\nu$
 is  {\!}    the   L\'evy  measure (or {\!}     intensity measure) of ${L};$    
   $(Q, a, \nu)$ is  {\!}    called the generating triplet (or {\!}     characteristics)   of ${L}$; it uniquely identifies  {\!}    the law of ${L}$. 
 
To {\!} study \eqref{25} we  may assume that $a {\!} = 0$ because eventually  we can replace the drift  $b(t,x)$ with   $b(t,x) + \int_0^t \sigma(s)\, a {\!} \, ds$. 

 The  Poisson random measure $N$ associated to {\!} ${L}$ is  {\!}    defined by 
$
N((0,t] \times  A) $ $ = \sum_{0 < s  {\!}    \le t} 1_{A} (\triangle {L}_s) 
$, 
 for {\!}     any Borel set $A \subset {{\R}^d}  \setminus  {\!}    \{ 0 \}$ with 
 $\triangle {L} _s$ $ = {L}_s  {\!}    - {L}_{s-}$. 
 
 According to {\!} \eqref{levii} with $a=0$ we have the following  L\'evy-It\^o {\!} path decomposition:

  There exists  {\!}    a {\!} $Q$-Wiener {\!}     
 process  {\!}    $B = (B_t)$ on $(\Omega, {{\mathcal F}}, \P)$ independent of $N$
with $d \times  d$ covariance 
matrix  $Q$ 
  such that on some almost sure event $\Omega'$ we have  
  \begin{equation} \label{itl}
 {L}_t =  A_t + B_t + C_t,
  \;\;\; t \ge 0, \;\;\; \text{with components}
\end{equation}  
 \begin{equation*}
 \label{ito1}
 A_t^j =
 \int_0^t \int_{\{ |x| \le 1\} } x_j \tilde N(dr, dx), 
 \;\;\; 
 C_t^j = \int_0^t \int_{\{ |x| > 1 \} } x_j  N(dr, dx),\;\; j =1, \ldots, d; 
 \end{equation*}
 here $\tilde N$ is  {\!}    the compensated Poisson measure (i.e., $\tilde N (dt,  dx)$ $ = N(dt, dx)-  dt \nu(dx)$).

Let us  {\!}    fix a {\!} deterministic Borel and locally bounded function ${\tilde \sigma}: [0,\infty) \to {\!} {\R}^{n} \otimes    {\R}^d$. 
The stochastic integral process  {\!}    ${\tilde M} = (\tilde M_t)$,
\begin{equation}\label{sto1}
\begin{array}{l}
\tilde M_t =\int_0^t {\tilde \sigma}(s) d{L}_s,\;\; t \ge 0,
\end{array}  
\end{equation}
is  {\!}    well defined; $\tilde M_t$ is  {\!}    a {\!} limit in probability of suitable Riemann-Stieltjes  {\!}    sums, see for {\!}     instance Chapter {\!}     2 in \cite{RS}  (we are considering the c\`adl\`ag  version of such stochastic integral).   
Equivalently, 
 one can define 
\begin{gather}\label{sto2}
\int_0^t {\tilde \sigma}(s) d{L}_s  {\!}    = I_t + J_t + K_t, 
 \\ \nonumber {\!}     
I_t^i {\!}  = \sum_{j =1}^d  \int_0^t {\tilde \sigma}_{ij}(s) dA_s^j ,\;\; J_t^i=  \sum_{j =1}^d  \int_0^t {\tilde \sigma}_{ij}(s) dB_s^j,\;\;  K_t^i=\sum_{j=1}^d \int_0^t  {\tilde \sigma}_{ij}(s) dC_s^j,   
\end{gather} 
 $i=1, \ldots, n$.
 The components  {\!}    of $I$ and $J$ are $L^2$-martingales  {\!}    and $K_t$ 
 is  {\!}    a {\!} Lebesgue-Stieltjes  {\!}    integral defined pathwise (recall that $(C_t^j)$
 is  {\!}    a {\!} compound Poisson process).
 
 The following result 
 will be useful.
 \begin{proposition}\label{s22} Assume 
 that $\E |{L}_1|^{\theta} < \infty$ for {\!}     {\!}     some $\theta  \in (0,1)$. Let $T>0.$ Then we have (cf. \eqref{sto2})  
\begin{gather} \label{servi}
\E [| I_t -  I_s  {\!}    |^2] +  \E [|J_t - J_s|^2] \le   C_T |t-s|,\;\;\; t,s  {\!}    \in [0,T],
\\ \nonumber {\!}     
  \E [|K_t -   K_s|^{\theta}] \le C_T |t-s|, \;\; s,t \in [0,T],
\\  \nonumber
  \E [\sup_{t \le T }  |\tilde M_t|^{\theta} ] < \infty.
\end{gather}
 \end{proposition}
\begin{proof} The fist estimate $\E [| I_t -  I_s  {\!}    |^2]$  $ \le C_T |t-s|$ is  {\!}    clear {\!}     by the It\^o {\!} isometry and the fact that ${\tilde \sigma}$ is  {\!}    bounded on $[0,T]$. 
  Similarly we have, using also {\!} Corollary 2.10 in \cite{Ku},  for {\!}     $0 \le s  {\!}    < t \le T$,    
\begin{gather*}
 \E [| J_t -  J_s  {\!}    |^2 ] = \sum_{i=1}^n  \sum_{j,k =1}^d\E  \big [ \int_s^t  \int_{ \{|x|\le 1 \} } {\tilde \sigma}_{ij}(r)x_j  {\tilde \sigma}_{ik}(r) x_k  \, dr {\!}     \nu (dx)\big ]
 \\ = \int_s^t \int_{\{ |x|\le 1 \} } |{\tilde \sigma} (r) x|^2 dr {\!}     \nu (dx) 
 \le C_T (t-s) \int_{\{ |x|\le 1 \} } | x|^2  \nu (dx).
\end{gather*}
Moreover,  applying  the Doob theorem we get 
 \begin{gather*}
\E [\sup_{t \le T }  |I_t|^{2} ] < \infty,\;\;\;  \E [ \sup_{t \le T }  |J_t|^{2} ] < \infty.
\end{gather*}
 It remains  {\!}    to {\!} consider {\!}     $(K_t)$. We find (see also {\!} pag. 231 in \cite{A})
 \begin{gather*}
  |K_t - K_s|^{\theta} = \Big | \int_s^t \int_{\{ |x|>1 \} } {\tilde \sigma}(r)x \,  N(dr,dx)   
   \Big|^{\theta}  
  =  
  \Big |   \sum_{s {\,}    < u \le t } {\tilde \sigma}(u)  \triangle {L}_u    \, 1_{ \{ |\triangle {L}_u| >1 
 \}} \Big|^{{\theta}}   
\\ \le 
  \sum_{s {\,}    < u \le t } |{\tilde \sigma}(u)  \triangle {L}_u|^{\theta}    \, 1_{ \{ |\triangle {L}_u| >1  \} }  
\le \| {\tilde \sigma}\|_{0, T}^{\theta}   \sum_{s {\,}    < u \le t }  |\triangle {L}_u|^{{\theta}} \,   1_{ \{ |\triangle {L}_u| >1 
 \} }    
 \end{gather*}
 since the random sum is  {\!}    finite for {\!}     any $\omega \in \Omega$ and ${\theta} \le 1$; $\| {\tilde \sigma}\|_{0, T} = \sup_{t \in [0,T]} \| {\tilde \sigma}(t)\|$.  On the other {\!}     hand (cf. Section 2.3.2 in \cite{A})  we know that 
\begin{equation*}
\begin{array}{l}
\E \big [ \sum_{s {\,}    < u \le t }  |\triangle {L}_u|^{{\theta}} \,   1_{ \{ |\triangle {L}_u| >1 
 \} } \big] = \E \big [  \int_s^t \int_{\{ |x| > 1 \} } |x|^{{\theta}}  N(dr, dx) \big] 
 \\
  = (t-s)  \int_{\{ |x| > 1 \} } |x|^{{\theta}} \nu (dx)  = C_{\theta} (t-s). 
\end{array} 
\end{equation*}
 In the last passage we have used 
  Theorem 25.3    in \cite{Sato}:    
 $\E [|{L}_1|^{{\theta}}] <  \infty$ is  {\!}    equivalent to {\!}  $\int_{\{ |x| > 1 \} } |x|^{{\theta}} \nu (dx) < \infty$. Thus  {\!}    we arrive at 
 \begin{equation*}
 \E [|K_t - K_s|^{\theta}]  \le C_T |t-s|.  
 \end{equation*} 
 Finally, arguing as  {\!}    before, 
 \begin{gather*}
 \E [\sup_{t \le T} |K_t |^{\theta}] \le \E \big[\sup_{t \le T}     \sum_{s {\,}    < u \le t } |{\tilde \sigma}(u)  \triangle {L}_u   |^{{\theta}}   \, 1_{ \{ |\triangle {L}_u| >1 
 \}} \big]
 \\
 \le \tilde C_T \E [ \sum_{0 < u \le T }  |\triangle {L}_u|^{{\theta}} \,   1_{ \{ |\triangle {L}_u| >1 
 \} } ]
 \\ = \tilde C_T  \E \big[ \int_0^T \int_{\{ |x| > 1 \} }  |x|^{{\theta}}  N(ds, dx)\big]
 = \tilde C_T \,  T   \int_{\{ |x| >1 \}} |x|^{{\theta}} \nu  (dx) \big] < \infty.
\end{gather*}   
  The proof is  {\!}    complete.
 \end{proof}
 
\def\ciao1 { 
\begin{remark} \label{ma1} {\em  We point out that Theorem \ref{main1} could be proved  for {\!}     more general additive ${\R}^n$-valued processes  {\!}     $M= (M_t)$ (non necessarily of the form \eqref{sto1}) under {\!}     the next assumptions. 

First recall  the L\'evy-It\^o {\!} decomposition for {\!}     additive processes  {\!}    (see  Chapter {\!}     4 in \cite{Sato}):
\begin{equation} \label{}
M_t = \alpha {\!} (t) + K_t + S_t,\;\;\; t \ge 0,
\end{equation}
where $\alpha(t)$ is  {\!}    a {\!} deterministic function, $K = (K_t) $ is  {\!}    an ${\R}^n$-valued Gaussian additive process, $S= (S_t)$ is  {\!}    an ${\R}^n$-valued additive pure jump process. Moreover,
$$
S_t = S_t^1 + S_t^2, \;\; t \ge 0,  
$$
where $S^1$ and $S^2$ are the small jumps  {\!}    part and the large jumps  {\!}    part respectively.
 We assume that  $\alpha {\!} : [0,\infty) \to {\!} {\R}^n$ is  {\!}    locally Lipschitz so {\!} that possibly  replacing $b(t,x)$ by $b(t,x) + \frac{d\alpha}{dt}(t)$ we may assume that $\alpha {\!} (t)$ is  {\!}    identically zero. Moreover {\!}     (cf. Proposition \ref{s22})  we assume that, for {\!}     any $T>0$, there exists  {\!}    $C_T>0$ and  $\theta  = \theta_T \in (0,1)$ such that 
\begin{gather*} 
\E [ | S_t^1 -  S_s^1 |^2] +  \E[ |K_t - K_s|^2] \le C_T |t-s|,\;\;\; t,s  {\!}    \in [0,T],
\\ 
  \E [|S_t^2 -   S_s^2|^{\theta}] \le C_T |t-s|, \;\; s,t \in [0,T], 
\;\;\; 
\text{ and} \;\;    
  \E [\sup_{t \le T }  |M_t|^{\theta} ] < \infty.
\end{gather*}
Under {\!}     the previous  {\!}    conditions  {\!}    on $M$ one can still prove Theorem \ref{main1} assuming  Hypothesis  {\!}    \ref{qq} and $b \in   L^{\infty}(0, T; C_{b}^{0,{\beta}}({\R}^n; {\R}^n))$, ${\beta} \in (0,1)$.  The proof remains  {\!}    almost the same. 
 } 
 \end{remark} 
} 

Let us  {\!}    fix a {\!} metric space $(\Lambda {\!} ,d)$.
 Given two {\!} stochastic processes  {\!}    ${U} = (U_t)_{t \in [0,T]}$
and $V = (V_t)_{t \in [0,T]}$ defined on ${(\Omega, {{\mathcal F}}, \P)}$ and with values  {\!}    in  $(\Lambda,d)$, we say that $U$ is  {\!}    a {\!} { modification} { or {\!}     version} of $V$ if  $U_t = V_t$, $\P$-a.s., for {\!}     any $t \in [0,T]$.

\smallskip As  {\!}    before  ${L} = ({L}_t)$ is  {\!}    a {\!} $d$-dimensional  L\'evy process  {\!}    defined on   $(\Omega, {{\mathcal F}}, \P)$.
Let  $b : [0,T] \times  {\R}^n \to {\!} {\R}^n$ and $\sigma  : [0,\infty) \to {\!} {\R}^n \otimes    {\R}^d$ be Borel and locally bounded functions  {\!}    and   $A \in {\R}^n \otimes    {\R}^n$. 
  Let  $s  {\!}    \in [0,T)$, $x \in {\R}^n$ and consider {\!}     the SDE \eqref{SDE}.
    
 We say that an ${\R}^n$-valued  stochastic process  {\!}    ${V^{s,x}}$ $=(V_t^{s,x}) =$ $(V_t^{s,x})_{t \in [s,T]}$ 
 defined on $(\Omega, {{\mathcal F}}, \P)$
is  {\!}    a {\!}  strong solution  starting from  $x$ at time $s$ (cf. \cite{Ku} and \cite{A}) if,   for {\!}     any $t \in [s,T]$, 
  $V_t^{s,x}: \Omega \to {\!} {\R}^n$ 
 is  {\!}    ${{\mathcal F}}_{s,t}^{L}$-measurable; further {\!}     one requires  {\!}    that there exists  {\!}    $\Omega_{s,x}$ (an almost sure event,   possibly depending also {\!} on $s$ and $x$ \ but independent of $t$) such that the next conditions  {\!}    hold for {\!}     any $\omega \in \Omega_{s,x}$: 
 \\ (i) the map:
  $t \mapsto {\!} V_t^{s,x}(\omega)$ is  {\!}    c\`adl\`ag on $[s,T]$;
\\   (ii)  
  we have, for {\!}     $t \in   [s  {\!}    , T]$, $\text{with}\; M_t = \int_0^t \sigma(r) d{L}_r,$ and $\tilde b (t,x) = b(t,x) + Ax$, 
\begin{equation} \label{SDE3} 
 {V}^{s,x}_t (\omega) =  x + \int_s^t \tilde b(r, {V}^{s,x}_r(\omega)) dr {\!}     + M_{t}(\omega) - M_s(\omega), 
\end{equation}
(iii) 
 the path $t \mapsto {\!} L_t(\omega)$ is  {\!}    c\`adl\`ag and $L_0(\omega)=0$.

Given a {\!} strong solution $V^{s,x}$ we set
for {\!}     any $0 \le t \le s$, $V_t^{s,x} =x$ on $\Omega$.

\medskip  We finish the section with a {\!}  simple lemma {\!} 
 about  the possibly degenerate SDE (cf. \eqref{d})
  \begin{equation} \label{66}
dX_t = A X_t dt + b(X_t)dt + C d L_t ,\;\; X_s  {\!}    =x.
\end{equation} 
with    $b : \R^n \to {\!} \R^n$ locally bounded, $A \in {\R}^n \otimes    {\R}^n$ and $C \in {\R}^n \otimes    {\R}^d$, $1 \le d \le n$. This  {\!}    result can be useful to {\!} check the validity of Hypothesis  {\!}    \ref{qq} for {\!}     SDEs  {\!}    like \eqref{66} (we will use this  {\!}    lemma {\!} in Section 5).   
  It says  {\!}    that from  existence of  strong solutions  {\!}    and corresponding $L^p$-estimates  {\!}    
 when  $s=0$ one can deduce existence and $L^p$-estimates  {\!}       
 when  $s  {\!}    \in (0,T)$.  The proof is  {\!}    similar {\!}     to {\!} the one of Theorem 6.6 in \cite{Pr}.

  \begin{lemma} \label{basic}  Let us  {\!}    consider {\!}      SDE \eqref{66} and fix $T>0$. 
   Suppose that for {\!}     a {\!} given Levy process  {\!}    $L$ with generating triplet $(Q, 0, \nu)$ (cf. \eqref{levii}) 
defined on some  probability space     $(\Omega, {\mathcal F}, \P)$,
for {\!}     given $b$ and $A$ and $C$, for {\!}     any $x \in \R^n$ there exists  {\!}    a {\!} unique strong solution $(X_t^x)$ $= (X_t^{0,x})$ to {\!} \eqref{66} on $[0,T]$ when $s=0$.

Suppose that given two {\!} strong solutions  {\!}    $(X^x_t)_{t \in [0,T]}$ and 
 \ $(X^y_t)_{t \in [0,T]}$
  of \eqref{66}  defined on   $(\Omega, {{\mathcal F}}, \P)$,  starting at  $x$ and $y \in {\R}^n$   respectively, 
      we have, for {\!}     $p \ge 2$, 
 \begin{equation} \label{ciao223}
\begin{array} {l}
\E\big[ \sup_{0 \le r {\!}     \le T} |X^x_r {\!}     - X_r^y|^p \big] \le C_T \,
 |x-y|^p,
\end{array}  
  \end{equation}
 with  $C_T = C \big( (Q, 0, \nu), $ $A$, $C$,    $ b ,
 n,  p, T \big) >0$. 
 
 Then   
 for {\!}     any $x \in {\R}^n$, $s  {\!}    \in [0,T)$, 
there exists  {\!}    a {\!}  unique strong solution   ${\hat X}^{s,x} 
=({\hat X}^{s,x}_t)_{t \in [0,T]}$ to {\!} \eqref{66} on   
 $(\Omega, {\mathcal F}, \P)$ (recall that ${\hat X}_t^{s,x} =x$ for {\!}     $t \le s$).
 Moreover,   if ${ V}^{s,x}$ and ${ V}^{s,y}$ are two {\!} strong solutions  {\!}     defined  $(\Omega, {{\mathcal F}}, \P)$     one has:      
 \begin{equation} \label{ciao22}
\sup_{s {\,}    \in [0,T]}  \, \E[\sup_{ s  {\!}    \le t \le T} |{V}_t^{s,x} - {V}_t^{s,y}|^p] 
 \le C_T \,
 |x-y|^p, \; \quad  x,\, y \in {\R}^n,  \;  p \ge 2. 
  \end{equation} 
  \end{lemma}
   \begin{proof} Define $\tilde b (x) = b(x) + Ax$, $x \in \R^n$. 
   \\
 {\sl Existence.}  Let us  {\!}    fix $s  {\!}    \in [0,T]$ and consider {\!}      the
 process  {\!}     ${ L^{(s)} = (L^{(s)}_t)} $ on $(\Omega, {\cal F}, \P)$, with 
  $L^{(s)}_t =L_{s+ t} - L_s$, $t \ge 0$. This  {\!}     is  {\!}    a {\!}  L\'evy process  {\!}    
  with the same generatig triplet    of $L$
and it is  {\!}    independent of ${\cal F}_s^L$ (cf. Proposition 10.7 in \cite{Sato}).
 We know that  there exists  {\!}    a {\!} unique strong solution on $(\Omega, {\cal F}, \P)$
to {\!}   
\begin{gather} \label{dedo}
X_t = x +  {\int_0^{t}}   \tilde  b(X_{l}) dl +  C L_{t}^{(s)},\; \quad  t \in [0,T],
\end{gather}
which we indicate by   $(X^x_{t, L^{(s)}})$
to {\!} remark its  {\!}    dependence on  ${ L^{ (s) }} $.
  For {\!}     any $t \in [0,T]$,  $X^x_{t, L^{(s)}}$ is  {\!}    measurable with respect to {\!} ${\cal F}^{L^{(s)}}_t = {\cal F}^L_{s,t +s}$. 
 Introduce  a {\!} new  process  {\!}    with c\`adl\`ag trajectories  {\!}    
 $({\hat X}_t^{s,x})_{t \in [0,T]}$,
\begin{gather} 
\label{strong1}
{\hat X}_t^{s,x} =  X_{t-s, L^{(s)}}^{x    },\;\; \text{for}   \;   s  {\!}    \le t \le T; \;\quad    {\hat X}_t^{s,x} =x,\quad  0\le t \le s.
\end{gather} 
Setting  $V_t = {\hat X}_t^{s,x}$, $t \in [0,T]$,  we have  that
 $V_t$ is  {\!}    ${\cal F}^L_{s,t}$-measurable, $t \ge s$. Further {\!}     it solves  {\!}     SDE \eqref{66}; indeed, for {\!}      $t \in [s,T]$, 
$$
V_t = X_{t-s, L^{(s)}}^{x    } = x + \int_0^{t-s}  \tilde b( X_{r, L^{(s)}}^{x    }) dr {\!}     +  C [L_{t} - L_s]
= x + \int_s^{t}  \tilde b(V_{r}) dr {\!}     +  C [L_{t} - L_s].  
$$
{\sl  Uniqueness.} Let $(V^{s,x}_t)$ 
be another {\!}     strong solution.
 We have, $\P$-a.s., for {\!}     $s  {\!}    \le t \le T$,
\begin{gather*}
V^{s,x}_{t-s  {\!}    +s} = \;  
   x + \int_s^{t}  \tilde b(V^{s,x}_p) dp \, + \, C[  L_{t} - L_s]
\\=
x + \, \int_0^{t-s}  \tilde b(V^{s,x}_{p+s}) dp \, + \,  C[ L_{t} - L_s]
= x + \, \int_0^{t-s}   \tilde b(V_{p+s}^{s,x}) dp \, + \,  C L_{t-s}^{(s)}.
\end{gather*}
Hence $(V^{s,x}_{r {\,}     +s})_{r {\,}     \in [0, T-s]}$ solves  {\!}    \eqref{dedo} 
on $[0, T-s]$. By \eqref{ciao223} we get 
\begin{gather*}
\P (V_{r {\,}     +s}^{s,x} = X_{r, L^{(s)}}^x, \, r {\!}     \in [0, T-s])= \P ( V_{r {\,}     +s}^{s,x} = \tilde X^{s,x}_{r {\,}     +s}, \, r {\!}     \in [0, T-s])=1.
\end{gather*}
This  {\!}    gives  {\!}     the assertion.    

\smallskip 
 \noindent  {\sl $L^p$-estimates.} We have for {\!}     any fixed $s  {\!}    \in [0,T]$,
  $\E[ \sup_{ s  {\!}    \le t \le T} |{V}_t^{s,x} - {V}_t^{s,y}|^p] 
$ $= 
\E[\sup_{ s  {\!}    \le t \le T}  |X^x_{t-s  {\!}    , \, L^{(s)}} - X^y_{t-s  {\!}    , \, L^{(s)} }|^p]$, $p \ge 2$, by uniqueness.    
Using   \eqref{ciao223} we get
\begin{gather*}
\sup_{s {\,}    \in [0,T]} \E[ \sup_{ s  {\!}    \le r {\!}     \le T} |{V}_r^{s,x} - {V}_r^{s,y}|^p] 
= 
\sup_{s {\,}    \in [0,T]} \E[\sup_{ s  {\!}    \le r {\!}     \le T}  |X_{r-s  {\!}    , \, L^{(s)}}^{x } - X_{r-s  {\!}    , \, L^{(s)} }^y|^p] 
\\
\le \sup_{s {\,}    \in [0,T]}\E[\sup_{r {\,}     \in [0,T]}  |X^x_{r {\,}     , \, L^{(s)}} - X^y_{r, \, L^{(s)} }|^p]  
\le C_T \,
 |x-y|^p. 
 \end{gather*}
 \end{proof}

\section{The main results}

We prove  an extension of  Theorem 5.1 in \cite{Pr} which allows  {\!}    to {\!} treat   SDEs  {\!}    of the form  \eqref{SDE}.  
This  {\!}    result implies  {\!}    Theorem \ref{main1}.  Recall that 
 in \cite{Pr} we have considered  \eqref{SDE} only when  $n=d$, $A =0$ and  $\sigma  = I$.

 We point out that the next statements  {\!}    (i)-(v) hold when $\omega $
 belongs  {\!}    to {\!}  $\Omega'$ (an almost sure event) which  is  {\!}     independent of $x $ $\in {\R}^n$, $s$, $s_0$, and $t $ $ \in [0,T]$. 
\begin{theorem} \label{h32} 
We consider {\!}       SDE \eqref{SDE}. Suppose that  
 $b \in L^{\infty}(0,T; C_b^{0, {\beta}}({\R}^n; {\R}^n))$, ${\beta} \in (0,1),$   $\sigma  : [0,\infty) \to {\!} {\R}^n \otimes    {\R}^d$  Borel and locally bounded, $A \in {\R}^n \otimes    {\R}^n$ and $L$  with generating triplet $(Q, 0, \nu)$ (cf. \eqref{levii}) 
 verify 
 Hypothesis  {\!}    \ref{qq}. Let $L$ be defined on  $(\Omega, {{\mathcal F}}, \P)$
  such that      $\E [|L_1|^{\theta}] < \infty$, for {\!}     some $\theta  \in (0,1)$. 
 
  Then there exists  {\!}     a {\!}  mapping   $\psi  (s,t,x, \omega)$, 
\begin{gather} \label{nuova}
\psi  : [0,T] \times  [0,T] \times  {\R}^n \times  \Omega \to {\!} {\R}^n,
\end{gather}
which is  {\!}    ${\mathcal B}([0,T] \times  [0,T]$ $ \times  {\R}^n) \times    {{\mathcal F}} $-measurable
and 
such that $( \psi(s,t,x , \cdot ))_{t \in [0,T]}$
is  {\!}    a {\!} strong solution of \eqref{SDE} starting from $x$ at time $s.$
 Moreover, there exists  {\!}     $\Omega'$ (almost sure event)
  such that 
 the following statements  {\!}    are satisfied  for {\!}     any $\omega \in \Omega'$.
  
 \smallskip \noindent   
(i)    For {\!}      $x \in {\R}^n$,  the mapping: 
 $s  {\!}    \mapsto {\!} \psi  (s,t,x, \omega)$ is  {\!}    c\`adl\`ag on $[0,T]$  (uniformly in $t$ and $x$), i.e., 
  let $s  {\!}    \in (0,T)$ and take  sequences  {\!}    $(s_j) $ and $(r_m)$ with   $s_j \to {\!} s^-$ and $r_m \to {\!} s^+$; we have, for {\!}     any $R>0$,  
\begin{gather} \label{sd}
 \lim_{m \to {\!} \infty} \sup_{|x| \le R}\sup_{t\,  \in [0,T]} |\psi  (r_m,t,x, \omega) - \psi  (s,t,x, \omega) | =0,
\\ \nonumber {\!}      
\lim_{j \to {\!} \infty} \sup_{|x| \le R}\sup_{t \, \in [0,T]} |\psi  (s_j,t,x, \omega) - \psi  (s-,t,x, \omega) | =0
\end{gather}
(similar {\!}     assertions  {\!}     hold also {\!} for {\!}     $s  {\!}    =0$ and  $s=T$). 
 \\   (ii) For {\!}     any $x \in {\R}^n$, $s  {\!}    \in [0,T]$,  
 $ \psi  (s,t,x, \omega) =x $ if $\, 0 \le t \le s$, and if $t \in [s,T],$
\begin{equation} \label{d566}
 \psi  (s,t,x, \omega) = x + \int_{s}^{t} [ b\left(r,  \psi  (s,r,x, \omega) \right) + A\psi  (s,r,x, \omega) ] dr\,  + M_t(\omega) - M_s(\omega). 
\end{equation} 
 (iii) For {\!}     any $s  {\!}    \in [0,T]$, the function $x \mapsto {\!} \psi  (s,t,x, \omega)$ is  {\!}    continuous  {\!}    in $x$ (uniformly in $t$). Moreover, 
 for {\!}     any  integer {\!}     $m > 2n$, there exists  {\!}    a {\!} ${\cal B}([0,T])$ $\times  {{\mathcal F}}$-measurable function  $V_m : [0,T] \times  \Omega $ $ \to {\!} [0, \infty]$ such that $\int_0^T V_m(s  {\!}    , \omega) ds  {\!}    < \infty$ and 
\begin{gather} 
\label{toro}
\sup_{t \in [0,T]} |\psi  (s,t,x, \omega)  - \psi  (s,t,y, \omega)| 
\\ \nonumber {\!}     
\le V_m(s, \omega) \,  |x-y|^{\frac{m - 2n}{m}} [(|x| \vee |y|)^{\frac{2n + 1}{m}} \vee 1],\;\;\; x,y 
\in {\mathbb R}^n, \, m > 2n , \; s  {\!}    \in [0,T].  
\end{gather} 
(iv) For {\!}     any $0 \le s  {\!}    <  r {\!}     \le t \le T$, $x \in {\mathbb R}^n$, it holds: 
\begin{gather} \label{fl2}
 \psi  (s  {\!}    ,t,x, \omega) = \psi  (r {\!}     ,t, \psi  (s  {\!}    ,r,x, \omega), \omega). 
\end{gather} 
(v) 
  Let $s_0 \in [0,T)$, $\tau = \tau(\omega) \in (s_0, T]$  and  $x \in {\mathbb R}^n$. If  a {\!}  c\`adl\`ag    function  
 $g:  [s_0,\tau ) \to {\!}  {\mathbb R}^n$ is  {\!}    a {\!} solution to {\!}  the following integral equation
\begin{gather} \label{integral} 
g(t) =  x + \int_{s_0}^{t} [b\left(r, g(r)\right) + A g(r)]  dr {\!}     
 +  M_{t} (\omega) - M_{s_0}(\omega),\;\;\; t \in [s_0,\tau), 
\end{gather}
then     $g(r) = \psi(s_0,r,x, \omega)$, for {\!}     $r {\!}     \in [s_0,\tau)$. 
\end{theorem}

Once we have proved Theorem \ref{h32}, one can easily  obtain an analogous  {\!}     of Corollaries  {\!}    5.4 and 5.5  in \cite{Pr} for {\!}     equation \eqref{SDE} when $b$ is  {\!}    possibly unbounded. The proofs  {\!}     remain the same and are based on a {\!} standard localization procedure. Let us  {\!}    first  state a {\!} result analogous  {\!}    to {\!}  Corollary 5.4.

\begin{corollary} \label{unb} (i)  Let us  {\!}    consider {\!}     \eqref{SDE}    with  a {\!} measurable  mapping $b: [0,T] \times  {\R}^n \to {\!} 
{\R}^n$     such that, for {\!}     any  function 
 $\rho {\!} \in C_0^{\infty} ({\R}^n)$, one has  {\!}     
 $$ 
    \begin{array}{l}
 b \cdot \rho {\!} \in L^{\infty}(0, T; C_{b}^{0, \beta}({\R}^n; {\R}^n)), \;\;\; \beta \in (0,1). 
\end{array}
$$
(ii)    Let $L$ be  an $\R^d$-valued L\'evy process  {\!}     such that $\E|L_1|^{\theta} < \infty$ for {\!}     some $\theta  \in (0,1)$.
\\   
(iii) Let $A \in \R^n \otimes    \R^n$ and $\sigma   : [0, \infty) \to {\!} \R^n \otimes    \R^d $  be locally bounded. Suppose that,  
 for {\!}     any $\eta \in  C_0^{\infty} ({\R}^n) $, the SDE  
$$
dZ_t = (\eta \cdot b)(t, Z_t) dt + A Z_t dt + \sigma(t) d{L}_t , \;\; t \in [s,T], \;\; Z_s  {\!}    =x \in {\R}^n,
$$
 satisfies  {\!}    Hypothesis  {\!}    \ref{qq}. 

Then there exists  {\!}     $\tilde \Omega$ (almost sure event) such that, for {\!}     any $x \in {\R}^n$, $\omega'' \in \tilde \Omega$,  
$s_0 \in [0,T)$ and ${\tau = \tau(\omega'') } \in (s_0, T]$, if $f_1 , f_2 :
 [s_0,\tau ) \to {\!}  {\R}^n$ are c\`adl\`ag solutions  {\!}    of 
\eqref{integral} when $\omega = \omega''$,  then
 $f_1(r) = f_2(r)$, $r {\!}     \in [s_0, \tau)$.
\end{corollary}
  
  One can also {\!}   construct  ``path by path''   strong solutions  {\!}    to {\!} \eqref{SDE} even when $b$ is  {\!}    possibly unbounded.     
 To {\!} simplify  we only consider {\!}        $s=0$. The next result  is  {\!}     the  analogous  {\!}    of Corollary 5.5 in \cite{Pr} (it can be proved with the same proof of \cite{Pr}).
\begin{corollary} \label{cons} Let us  {\!}    consider {\!}     \eqref{SDE}. 
 Suppose  that   assumptions  {\!}    (i), (ii) and (iii) of Corollary \ref{unb} hold. Moreover {\!}     assume that there exists  {\!}    $C_0  >0$ such that  
 $$  
|b(t,x) | \le C_0 (1+ |x|),\;\; x \in {\R}^n, \quad  t \in [0,T],
$$  
Let  $x \in \R^n, s=0$. Then   
there exists  {\!}     a {\!} (unique) strong solution to {\!} \eqref{SDE} starting at $x$. This  {\!}    strong solution can be constructed in a {\!} deterministic way, arguing for {\!}       each $\omega \in \Omega$, $\P$-a.s. 
\end{corollary}

  In order {\!}     to {\!} prove Theorem \ref{h32} we start with a {\!} lemma {\!} which gives  {\!}    an integration-by-parts  {\!}    formula. We use  $e^{tA} = \sum_{k=0}^{\infty}\frac{t^k A^k}{k!} $.
   \begin{lemma}
   \label{sa1} Let $\sigma(t)$, $A$ and $L$ as  {\!}    in Theorem \ref{h32}. Recall that $M_t = \int_0^t \sigma(u) dL_u$. 
  We have, $\P$-a.s., 
\begin{gather}\label{by1}
\int_s^t e^{(t-r)A } A  M_{r} dr
 = -\int_s^t \frac{d}{dr}[e^{(t-r)A }] \, M_{r}  dr
\\ \nonumber {\!}       =   \int_s^t e^{(t-r)A } \sigma(r){dL}_r {\!}     +  e^{[t-s]A} M_{s} \, - \, M_{t}, \;\;\;  0 \le s  {\!}    \le t. 
\end{gather}  
 \end{lemma}
 \begin{proof} We use a {\!} stochastic Fubini {\!}  theorem (see Theorem 4.7 in \cite{Sato1}
  or {\!}     Proposition 2.7 in \cite{RS}). We find, $\P$-a.s.,  
 \begin{gather*}
\int_0^s  {\!}    e^{(t-r)A } A  \Big (\int_0^r {\!}     \sigma(u) d{L}_{u} \Big) dr {\!}      
   =
 \int_0^s  {\!}     \Big ( \int_u^s  {\!}    e^{(t-r)A } A dr {\!}     \Big ) \sigma(u) d{L}_{u}
 \\ = - \int_0^s  {\!}     \Big ( \int_u^s  {\!}    \frac{d}{dr}[e^{(t-r)A }]  dr {\!}     \Big ) \sigma(u) d{L}_{u}
 \\  = 
  \int_0^s  {\!}      [e^{(t-u)A } - e^{(t-s)A } ]   \sigma(u) d{L}_{u},
 \;\; t > 0,\; 0 \le s  {\!}    \le t.
\end{gather*} 
 and  so {\!} 
$\int_s^t e^{(t-r)A } A  \Big (\int_0^r {\!}     \sigma(u) d{L}_{u} \Big) dr
  =
 \int_s^t  e^{(t-u)A }    \sigma(u) d{L}_{u} +    e^{(t-s)A } M_s  {\!}       - M_t.$
  \end{proof}

In order {\!}     to {\!} prove Theorem \ref{h32} with the approach of  \cite{Pr}   it is  {\!}    useful   
  to {\!} pass  {\!}    from   SDE \eqref{SDE} to {\!}  the following modified SDE with bounded coefficients  {\!}    in which $A$ is  {\!}    not present
\begin{equation}\label{dai}
U_{t}(\omega) =  x + \int_{s}^{t}  \tilde b (v, U_{v}(\omega))   dv \, + \, \tilde M_{t}(\omega) - \tilde M_s(\omega),   
\end{equation}
with   
\begin{gather} \label{see}
\tilde b (r, x) = e^{-rA} b(r,e^{rA} x), \;\;\; r {\!}     \in [0,T], \, x \in {\R}^n.
\\  \nonumber {\!}     
 \tilde M_{t} = \int_0^t e^{-rA} \sigma(r) dL_r,\;\; t \ge 0.
\end{gather}
 \begin{lemma}\label{qee} Let us  {\!}    fix $s  {\!}    \in [0,T)$ and $x \in \R^n$. If $Z(t) = Z_{t}^{s,e^{sA}x}$ is  {\!}    a {\!} strong solution to {\!} \eqref{SDE} then 
 \begin{equation}\label{d3}
 U(t) =  e^{-tA} Z(t)
\end{equation}
 is  {\!}    a {\!} strong solution to {\!} \eqref{dai}  starting from $x$ at time $s$. Viceversa, If $U(t) = U_{t}^{s,e^{-sA}x}$ is  {\!}    a {\!} strong solution to {\!} \eqref{dai} then 
  $
 Z(t) =  e^{tA} U(t) 
 $
 is  {\!}    a {\!} strong solution to {\!} \eqref{SDE} starting from $x$ at time $s$.
 
 In particular {\!}      SDE  \eqref{SDE} verifies  {\!}    Hypothesis  {\!}    \ref{qq} if and only if  SDE \eqref{dai} (with $\tilde b$, $A=0$, $\tilde \sigma(u) = e^{-uA} \sigma(u)$) verifies  {\!}    Hypothesis  {\!}    \ref{qq}. 
\end{lemma}   
\begin{proof}  
 Assume that   
    $Z(t) = Z_{t}^{s,x}$ is  {\!}    a {\!} strong solution to {\!} \eqref{SDE}.
 Hence, $\P$-a.s.,
 \begin{gather} \label{s0}
  Z(t) = e^{sA} x  +  \int_s^t  b(r, Z(r))dr {\!}     + \int_s^t A Z(r) dr {\!}     +  {M}_t - {M}_s  {\!}    , \;\; t \ge s.  
\end{gather}
Now define $H(t) = Z(t) - {{M}}_{t}$. We find 
\begin{gather*}
 H(t) =e^{sA} x  +  \int_s^t  b(r, H(r) +  {{M}}_{r} )dr {\!}     + \int_s^t A  H(r) dr {\!}     +  \int_s^t A {{M}}_{r} dr {\!}     - {{M}}_s.
\end{gather*} 
Hence 
\begin{gather*}
H'(t) = \frac{d}{dt}H(t) =  b(t, H(t) +  {{M}}_{t} ) + A  H(t)  +  A {{M}}_{t}, \;\; t \in  ]s,T],
 \\ H(s) =e^{sA} x - {{M}}_s. 
\end{gather*} 
It follows  {\!}    that  
\begin{gather*}
H(t) = e^{[t-s]A} e^{sA} x  +  \int_s^t e^{(t-r)A} b(r, H(r) +  {{M}}_{r})dr {\!}     + \int_s^t e^{(t-r)A } A {{M}}_{r}dr {\!}     \\
\,   - \, e^{[t-s]A}{{M}}_s.
\end{gather*}
Using Lemma {\!} \ref{sa1} we get 
$$
\int_s^t e^{(t-r)A } A  {{M}}_{r} dr=   \int_s^t e^{(t-r)A } \sigma(r){dL}_r {\!}     +  e^{[t-s]A} M_{s} \, - \, M_{t}; 
$$
we obtain
 \begin{gather*}
Z(t) = e^{tA} x  +  \int_s^t e^{(t-r)A} b(r, Z(r))dr {\!}     + \int_s^t e^{(t-r)A }   \sigma(r){dL}_r,
\end{gather*}
and so {\!} 
\begin{gather*}
 e^{-t A} Z(t) =    x  +   \int_s^t e^{-r {\!}     A} b(r,e^{rA} [e^{-r {\!}     A}Z(r)])dr {\!}     + \int_s^t e^{-rA } \sigma(r){dL}_r.  
\end{gather*}
 Hence
 \begin{gather*}
 U(t) =    x  +   \int_s^t e^{-r {\!}     A} b(r,e^{rA} U(r))dr {\!}     +  \tilde M_t - \tilde M_s.  
\end{gather*}
Viceversa, if  we start with a {\!} strong solution $U(t) = U_{t}^{s,e^{-sA}x}$
then 
 \begin{gather*}
 U(t) =  e^{-sA}  x  +   \int_s^t e^{-r {\!}     A} b(r,e^{rA} U(r))dr {\!}     + 
  \int_s^t e^{-r {\!}     A} \sigma(r) dL_r,\;\; t \ge s.
\end{gather*}
Applying $e^{tA}$ to {\!} both sides  {\!}    we get that $Z(t) = e^{tA} U(t)$ is  {\!}    a {\!} strong solution to {\!} 
$$
dZ_t = b(t, Z_t) dt + A Z_t dt + \sigma(t) d{L}_t , \;\; t \in [s,T], \;\; Z_s  {\!}    =x. 
 $$  
  It is  {\!}    not difficult 
to {\!} check the second assertion about Hypothesis  {\!}    \ref{qq}. We  only note that when
 Hypothesis  {\!}    \ref{qq} holds  {\!}    for {\!}     \eqref{SDE} then concerning equation \eqref{dai}        we have, for {\!}     $T>0$,
 \begin{gather*}
 \E [ \sup_{s {\,}    \le t \le T}| U_{t}^{s,x} -  U_{t}^{s,y} |^p]
 =
 \E [ \sup_{s {\,}    \le t \le T}|e^{-tA} [ Z_{t}^{s,e^{sA}x} -  Z_{t}^{s,e^{sA}y}]  |^p]
\\ 
\le C_T |e^{sA}x -  e^{sA}y|^p \le \tilde C_T |x -  y|^p, 
\end{gather*}
 where $\tilde C_T$ is  {\!}    independent of $s$, $x$ and $y$.
 Thus  {\!}    one  can easily prove that  Hypothesis  {\!}    \ref{qq} holds  {\!}     for {\!}     equation \eqref{dai}  as  {\!}    well.
  \end{proof} 
 According to {\!} the previous  {\!}    lemma {\!} in order {\!}     to {\!} prove Theorem \ref{h32} it is  {\!}     enough  to {\!} establish the next result.
   \begin{theorem} \label{h322} Assume 
 $b \in L^{\infty}(0,T; C_b^{0, {\beta}}({\R}^n; {\R}^n))$, $\sigma(t)$,    $A $, $L$ and $\theta$ 
as  {\!}    in Theorem \ref{h32} (hence Hypothesis  {\!}    \ref{qq} concerning equation \eqref{SDE} is  {\!}    verified).  
 Consider {\!}     the modified SDE  \eqref{dai} 
 with corresponding coefficients  {\!}    $\tilde b$   and $\tilde  \sigma  $  given in \eqref{see} and with the same L\'evy process  {\!}    $L$.
   
 \smallskip
{\sl Then  all the assertions  {\!}    (i)-(v) listed   in Theorem \ref{h32}  
hold for {\!}     equation \eqref{dai}. For {\!}     instance, there exists  {\!}    an almost sure event $\Omega'$ such that  for {\!}      $\omega \in \Omega'$ we have the following property:
   let  $s_0 \in [0,T)$, $\tau = \tau(\omega) \in (s_0, T]$  and  $x \in {\R}^n$; there exists  {\!}    a {\!} unique   c\`adl\`ag    function  
 $g:  [s_0,\tau ) \to {\!}  {\mathbb R}^n$ which solves  {\!}    the integral equation
$$ 
g(t) =  x + \int_{s_0}^{t} \tilde b\left(r, g(r)\right)  dr {\!}     
 +  \tilde M_{t} (\omega) - \tilde M_{s_0}(\omega),\;\;\; t \in [s_0,\tau), 
 $$
  }  
\end{theorem}
   
 {\sl  We stress  {\!}     that   under {\!}     the assumptions  {\!}    of Theorem \ref{h322}    we know  that    
   Hypothesis  {\!}    \ref{qq}  holds  {\!}     for {\!}     the modified  equation \eqref{dai} with coefficients  {\!}    $\tilde b$ and $\tilde \sigma  $ (see Lemma {\!} \ref{qee}). }   
  
\smallskip      
In  Section 4 we concentrate on the proof of Theorem \ref{h322}. The main problem with respect to {\!} \cite{Pr}  is  {\!}    that $\tilde M$,
 \begin{equation} \label{mm}
 {\tilde M}_t = \int_0^t e^{-rA} \sigma(r) dL_r,\;\; t \ge 0, \;\; \text{has  {\!}    not stationary increments  {\!}    in general}
\end{equation}
(in Theorem 5.1 of \cite{Pr} a {\!} SDE  like   \eqref{dai} is  {\!}    considered assuming that   $\tilde M$ is  {\!}     a {\!} L\'evy process).

\section{The proof of Theorem \ref{h322}  } 

We will consider {\!}     the  steps  {\!}    of the proof of the corresponding  Theorem 5.1 in \cite{Pr}  (see in particular {\!}     Sections  {\!}    3 and 4 in \cite{Pr}) showing that 
they still work without  the stationarity of  increments  {\!}    of the driving  process. To {\!} this  {\!}    purpose we have to {\!}    modify  the proofs  {\!}    of Lemma {\!} 4.3 and Theorem 4.4 in \cite{Pr}    using  Proposition     \ref{s22}.

\smallskip 
  We start with a {\!} 
 strong solution $({V}_t^{s,x})_{t \in [0,T]}$ to {\!} \eqref{dai} defined on $(\Omega, {{\mathcal F}}, \P)$
and  introduce the  $n$-dimensional process  {\!}     
 $  {\bar      Y}^{s,x} = (  {\bar      Y}^{s,x}_t)_{t \in [0,T]}$,  
\begin{eqnarray}
\label{yy1}
  {\bar      Y}_t^{s,x} = {V}_t^{x,s} - ({\tilde M}_t - {\tilde M}_s).
\end{eqnarray} 
Note that on  $\Omega_{s,x}$ (an almost sure event independent of $t$)  we have (cf. \eqref{see})  
\begin{equation} 
\label{yry}
  {\bar      Y}_t^{s,x} = x + \int_s^t {\tilde b}(r,   {\bar      Y}_r^{s,x} + ({\tilde M}_r {\!}     - {\tilde M}_s)) dr, \;\; t \ge s,
\end{equation} 
 and  $  {\bar      Y}_t^{s,x} = x$ on $\Omega $ if $t\le s$.  It follows  {\!}    that  $(  {\bar      Y}_t^{s,x})_{t \in [0,T]}$ have   continuous  {\!}    trajectories.  

Let us  {\!}    fix $s  {\!}    \in [0,T]$ and $x \in {\R}^n$. Setting    ${\bar      Y}_t^{s,x}(\omega) =0$, for {\!}     $t \in [0,T]$, if $\omega \not \in \Omega_{s,x}$, we find that 
 ${\bar      Y}^{s,x}_{\cdot } (\omega)$ $ \in {G_0} = C([0,T]; {\mathbb R}^n)$, for {\!}     any $\omega \in \Omega$. 
 Moreover {\!}       
\begin{eqnarray} \label{ancora}
{\bar      Y}^{s,x} = {\bar      Y}^{s,x}_{\cdot} \;\; \text{ is  {\!}    a {\!} random variable with values  {\!}    in 
${G_0} = C([0,T]; {\R}^n) $};  
\end{eqnarray}
 $G_0$ is  {\!}    the Banach space of all continuous  {\!}    functions  {\!}    from $[0,T]$ into {\!} $\R^n$  endowed with the supremum norm $\| \cdot \|_{G_0}$.  
Now, for {\!}     each fixed $s  {\!}    \in [0,T]$, we   obtain  a {\!} suitable version  of the random field $({\bar  Y^{s,x}})_{x \in {\R}^n}$  which takes  {\!}     values  {\!}    in $G_0$.   
  
The next result can be proved as  {\!}      Lemma {\!} 3.2 in \cite{Pr} (in  the corresponding proof  in \cite{Pr} the  stationarity of  increments  {\!}    of the driving L\'evy process  {\!}     was  {\!}    not used).

\begin{lemma} \label{fi}  Consider {\!}     \eqref{dai}.      
Let us  {\!}    fix $s  {\!}    \in [0,T]$  and consider {\!}     $ {\bar      Y}^{s} = ({\bar      Y}^{s,x})_{x \in {\R}^n}$ which takes  {\!}     values  {\!}    in ${G_0}$ (cf. \eqref{ancora}). We have:

\hh (i) There exists  {\!}    a {\!} continuous  {\!}    $G_0$-valued modification $Y^s  {\!}    =({ Y}^{s,x})_{x \in {\R}^n}$  (i.e., for {\!}     any $x \in {\R}^n,$ $ {\bar      Y}^{s,x}=   { Y}^{s,x}$ in $G_0$ on some almost sure  event).

\hh (ii)  For {\!}     any  $p >2n$ there exists  {\!}    a {\!}   r.v. ${U}_{s,p} $
 with values  {\!}    in $[0, \infty]$ such that, for {\!}     any $\omega \in \Omega$, $x, y \in {\R}^n$,
\begin{equation}
\label{est23}
\|   Y^{s,x}(\omega)   -   Y^{s,y} (\omega)  \|_{G_0} \le {U}_{s,p}(\omega) \, [(|x| \vee |y|)^{\frac{2n + 1}{p}} \vee 1] \,    |x-y|^{1- 2n/p}.
\end{equation} 
Moreover, 
\begin{gather} \label{ri4}
\sup_{s {\,}    \in [0,T]} \, \
\E [{U}_{s,p}^p]  < \infty.  
\end{gather} 
(iii) On some almost sure event $\Omega_s'$ (which is  {\!}    independent of $t$ and $x$) one has:
\begin{equation}
\label{yrye}
  {Y}_t^{s,x} = x + \int_s^t {\tilde b}(r,   {Y}_r^{s,x} + ({\tilde M}_r {\!}     - {\tilde M}_s)) dr, \;\; t \ge s,\; x\in {\mathbb R}^n
\end{equation}
 (where $ {Y}_t^{s,x}(\omega) =  ({Y}_{\cdot}^{s,x}(\omega))(t)$, $t \in [0,T])$;  this  {\!}    implies  {\!}    that, for {\!}     any $\omega \in \Omega_s'$, $x \in {\R}^n$, the map: 
  $t \mapsto {\!} {Y}_t^{s,x}(\omega)$ is  {\!}    continuous  {\!}     on $[0,T]$. 
\end{lemma}

Let $s  {\!}    \in [0,T]$. According to {\!} the previous  {\!}    result starting from $Y^s  {\!}    = (Y^{s,x})_{x \in {\R}^n}$ we  can define   random variables  {\!}    $X_t^{s,x} : \Omega \to {\!} {\mathbb R}^n$ as  {\!}    follows: $X_t^{s,x} = x$   if $t \le s$ and
\begin{equation}
\label{yy12}
X_t^{s,x} = Y_t^{x,s} + ({\tilde M}_t - {\tilde M}_s),\;\;\; s,t \in [0,T],\; x \in {\mathbb R}^n, \; s  {\!}    \le t.
\end{equation} 
By the properties  {\!}    of $Y^{s,x}$ 
we get   $\P ({\tilde X}^{s,x}_t = {X}^{s,x}_t ,\; t \in [0,T]) =1$, for {\!}     any $x \in {\mathbb R}^n$.

 Moreover, using also {\!}  \eqref{yrye}, we find that   for {\!}     some almost sure event $\Omega_{s}'$ (independent of $x$ and $t$) the map:
$t \mapsto {\!} {X}_t^{s,x}(\omega)$ is  {\!}    c\`adl\`ag  on $[0,T]$, for {\!}     any $\omega \in \Omega_s'$, $x \in {\R}^n$, and on
$\Omega_s'$ we have  
\begin{equation} \label{ssd}
{X}^{s,x}_t =  x + \int_s^t {\tilde b}(r, {X}^{s,x}_r) dr {\!}     + {\tilde M}_{t} - {\tilde M}_s,
\; s  {\!}    \le t \le T,\; x \in {\R}^n.
\end{equation}
Hence $({X}^{s,x}_t)_{t \in [0,T]}$ is  {\!}    a {\!}  particular {\!}     { strong solution} to {\!} \eqref{dai}.        
By Lemma {\!} \ref{fi} we also {\!} have, for {\!}     any $s  {\!}    \in [0,T]$, $x \in {\R}^n$,    
  $\lim_{y \to {\!} x} \sup_{t \in [0,T]}|  
X_t^{s,x}(\omega) - X_t^{s,y}(\omega) | =0,$ $\omega \in \Omega$. 
 
 \smallskip
The following  flow property can be proved as  {\!}      Lemma {\!} 3.3 in \cite{Pr} (indeed the corresponding  proof in \cite{Pr} does  {\!}     not use the stationarity of increments  {\!}    of the driving L\'evy process). 
 \begin{lemma} \label{fll}  Consider {\!}     the strong solution   $(X^{s,x}_t)_{t \in [0,T]}$ of \eqref{dai} defined in \eqref{yy12}.  
Let $0 \le s  {\!}    <u \le  T$. There exists  {\!}    $\Omega_{s,u}$ (an almost sure event  independent of $x \in {\R}^n$ and $t \in [0,T]$)  such that 
for {\!}     any $\omega \in \Omega_{s,u}$, $x \in {\R}^n$, we have 
\begin {gather} \label{co4} 
{ X_t^{s,x } (\omega)}  =  
 {X_t^{u,  \, X_{u}^{s,x}(\omega)} } \,(\omega) ,\;\;\; t \in [u,T], \;\; x \in {\R}^n. 
\end{gather}
\end{lemma}
 \noindent  Now as  {\!}    in \cite{Pr} we introduce  
 $ C({\R}^n; {G_0})$
 the space  of all functions  {\!}    from ${\R}^n$ into {\!} ${G_0}= C([0,T]; {\R}^n)$ endowed with the   compact-open topology.   
This  {\!}    is  {\!}    a {\!} complete separable metric space  endowed with the metric 
\begin{gather} \label{metr}
{d_0} (f,g) = \sum_{k \ge 1} \frac{1}{2^k} \frac{\sup_{|y|\le k} \|f(y) - g(y) \|_{G_0} }{1+ \sup_{|y|\le k} \|f(y) - g(y) \|_{G_0} },\;\;\; f,g \in C({\R}^n; {G_0}).
\end{gather}
We will also {\!} use  the continuous  {\!}     projections:  
\begin{gather} \label{pro2} 
 \pi_x : \;  { C({\R}^n; {G_0})} \to {\!} {G_0}, \quad  \pi_x (l) = l(x) \in G_0, \;\; x \in {\R}^n, \;\; 
 l \in C({\R}^n; {G_0}).  
\end{gather}
   By Lemma {\!} \ref{fi} for {\!}     any $s  {\!}    \in [0,T]$  the  random field ${ (Y^{s,x})_{x \in {\R}^n}}$ has  {\!}    continuous  {\!}    trajectories. 
It is  {\!}    straightforward  to {\!} prove that, 
 for {\!}     any $s  {\!}    \in [0,T]$, the mapping:
\begin{equation}
\label{cte}
 \omega  \mapsto {\!}  Y^s  {\!}    (\omega) = Y^{s,  \, \cdot \,} (\omega)
\end{equation} 
is  {\!}    measurable from $(\Omega, {{\mathcal F}}, \P) $ with values  {\!}    in $C({\R}^n; {G_0})$ (cf. page 702 in \cite{Pr}).
    
\smallskip
We will set  $Y = (Y^s)_{s {\,}    \in [0,T]}$ to {\!} denote the previous  {\!}    stochastic
process  {\!}    with values  {\!}    in $C({\mathbb R}^n; {G_0})$ and defined on $(\Omega, {{\mathcal F}}, \P)$.

The next two {\!} results  {\!}    correspond to {\!} Lemma {\!} 4.3 and Theorem 4.4 in \cite{Pr} respectively. In their {\!}     proofs  {\!}     the stationarity of increments  {\!}    of the driving L\'evy process  {\!}     has  {\!}    been used.   To {\!} overcome this  {\!}    difficulty  we will use  Proposition \ref{s22}.  
 \begin{lemma}  \label{modi1a} 
The  process  {\!}    $Y= (Y^s)$ with values  {\!}    in 
$C({\mathbb R}^n; {G_0})$ (see (\ref{cte})) is  {\!}    continuous  {\!}    in probability. 
 \end{lemma}
 \begin{proof} 
To {\!} perform the  proof of Lemma {\!} 4.3 in \cite{Pr} 
  replacing the L\'evy process  {\!}    $L$ of \cite{Pr}  with the additive process  {\!}    $\tilde M$ (and, moreover, $b$ with $\tilde b$ and  the dimension $d$ with $n$),
  we start to {\!}  
  choose $\beta$ small enough such that 
 \begin{gather} \label{bet}
 \beta (2n + 1)\,  < 2n {\theta}
\end{gather}  
 (cf. (4.5) in \cite{Pr} and recall that $C^{0, \beta}_b(\R^n; \R^n) \subset $ $ 
 C^{0,\gamma}_b(\R^n; \R^n)$ for {\!}     $0 < \gamma {\!} \le \beta \le 1$).  
  
Then we  replace the estimate after {\!}     (4.12) in \cite{Pr} as  {\!}    follows  {\!}    (cf. \eqref{mm} and see Proposition \ref{s22}):
\begin{gather*}
\E [| {\tilde M}_{s_n} - {\tilde M}_s|^{\frac{\beta (2n + 1) }{2n}}] 
 \le  \E \big[ \big |\int_s^{s_n} e^{-rA} \sigma(r) dL_r {\!}      \big |^{\frac{\beta (2n + 1) }{2n}} \big ]
\\ 
\le   \E [ \sup_{s {\,}    \in [0,T]} |{\tilde M}_s|^{\frac{\beta (2n + 1) }{2n} }]
< \infty.  
\end{gather*}
 The remaining part of the proof of  Lemma {\!} 4.3 in \cite{Pr} can be easily adapted to {\!} the present setting and we obtain the assertion.
\end{proof} 
 
 \begin{theorem}  \label{modi12}
Consider {\!}     the  process  {\!}    $Y = (Y^s)$ with values  {\!}    in 
$C({\R}^n; {G_0})$ (see (\ref{cte})). 
There exists  {\!}    a {\!} modification $Z = (Z^s)$ of $Y$  with c\`adl\`ag paths. 
\end{theorem}
 \begin{proof}  We follow the proof of the corresponding Theorem 4.4 in \cite{Pr} replacing $L$ in \cite{Pr} with $\tilde M$, the dimension $d$ with $n$ and $b$ with $\tilde b$. 
  In the sequel we only indicate some changes.  
 
The main changes  {\!}    are in Step IV of the proof of Theorem 4.4 in \cite{Pr}. 
 We   fix $ p \ge 32 n$ (i.e., $1- \frac{2n}{p} \ge 15/16$)  such that $\frac{8 (2n +1)}{p} < \frac{{\theta}}{4}$   
and consider {\!}     the r.v.
$$
 Z = 1 + \sup_{r {\,}     \in [0,T]} |\tilde M_r|^{\frac{8 (2n + 1) }{p}}.
$$
 One has  {\!}    $|{\tilde M}_{s'} - {\tilde M}_s|^{\frac{8 (2n + 1) }{p}}
 $ $\le 2 Z$, for {\!}     $0 \le s  {\!}    < s'$ $ \le T$. By Proposition \ref{s22}  we know that
$\E [Z^4]$ $ < \infty.$

Following the proof in   Step IV, replacing  $L$ with ${\tilde M}$ we   arrive at the problem of  estimating    $\Gamma_i$, $i=1, \ldots  {\!}    4$.  
 
The term  $\Gamma_4$ can be treated as  {\!}    in \cite{Pr}. Let us  {\!}    deal with the other {\!}     terms; setting 
 $\rho {\!} =  s_3 - s_1$  (recall that $0 \le s_1 < s_2 < s_3 \le T$) and using the r.v. $U_{s,p}$ of Lemma {\!} \ref{fi}, we consider {\!}     
 \begin{gather*}
 \Gamma_1 =  \E [ 1_{ \{    |{\tilde M}_{s_2}- {\tilde M}_{s_1}| >    |s_2 -s_1|^{1/8}   \}}  \cdot  1_{ \{    |{\tilde M}_{s_3}- {\tilde M}_{s_2}| >    |s_3 -s_2|^{1/8}   \}} ],
\\ 
\Gamma_2 =  \rho^{1- \frac{2n}{p}} [ \P ( |{\tilde M}_{s_3}- {\tilde M}_{s_2}| >    |s_3 -s_2|^{1/8} +  \P ( |{\tilde M}_{s_2}- {\tilde M}_{s_1}| >    |s_2 -s_1|^{1/8} )], 
\\
\Gamma_3 = \rho^{1- \frac{2n}{p}}  \E [1_{ \{    |{\tilde M}_{s_3}- {\tilde M}_{s_2}| >   |s_3 -s_2|^{1/8}   \}}
 Z\,  U_{s_2,p }^{8}  + 
1_{ \{    |{\tilde M}_{s_2}- {\tilde M}_{s_1}| >   |s_2 -s_1|^{1/8}   \}}
Z\,  U_{s_3,p }^{8} ].  
\end{gather*}
 We need to {\!} estimate, for {\!}       $r {\!}     \le s  {\!}    $, $r,s  {\!}    \in [0,T] $,   
\begin{gather} \label{w} 
\P ( |{\tilde M}_{s} - {\tilde M}_r {\!}     | >    |r-s|^{1/8} ).   
\end{gather}
We use Proposition \ref{s22} with its  {\!}    notation. 
 We have 
\begin{gather*}
 \P ( |{\tilde M}_{s} - {\tilde M}_r {\!}     | >     |r-s|^{1/8} ) \le
  \P ( |I_{s}- I_r| >     |r-s|^{1/8}/3 )
  \\ + 
   \P ( |J_{s} - J_r| >     |r-s|^{1/8}/3 )
 +  \P ( |K_{s} - K_r| >     |r-s|^{1/8}/3 ).
\end{gather*}
  Applying the Chebychev inequality we get
\begin{gather} \label{vaii}
 \P (  |{\tilde M}_{s} - {\tilde M}_r {\!}     | >     |r-s|^{1/8} ) \le
 \frac{9}{ |r-s|^{1/4}} \E [ |I_{s} - I_r {\!}     |^2 +  |J_{s} - J_r {\!}     |^2] 
 \\ \nonumber
 + 
    \frac{3^{{\theta}}}{ |r-s|^{{\theta}/8}} \E [| |K_{s} - K_r {\!}     ||^{{\theta}}]    
\le c_3  ( |r-s|^{3/4} +  |r-s|^{1- \frac{{\theta}}{8}}).    
\end{gather}
By \eqref{vaii}   we  estimate ${\Gamma_2}$
and ${\Gamma_3}$ as  {\!}    follows  {\!}    
 \begin{gather*}
{\Gamma_2} \le  \rho^{1- \frac{2n}{p}} [ \P ( |{\tilde M}_{s_3} - {\tilde M}_{s_2}| >    |s_3 -s_2|^{1/8}) +  \P ( |{\tilde M}_{s_2}  - {\tilde M}_{s_1}| >    |s_2 -s_1|^{1/8} )] 
\\
\le 2 c_3 \rho^{1- \frac{2n}{p}} (\rho^{3/4} + \rho^{1- \frac{{\theta}}{8}}). 
\\
{\Gamma_3} \le  \rho^{1- \frac{2n}{p}} (\E [Z^4])^{1/4}  
(\sup_{s {\,}    \in [0,T]} \E [U_{s,p }^{32}])^{1/4} 
 \, \Big [ (\P ( |{\tilde M}_{s_3} - {\tilde M}_{s_2}| >    |s_3 -s_2|^{1/8})^{1/2} 
\\ +  (\P ( |{\tilde M}_{s_2} -{\tilde M}_{s_1}| >    |s_2 -s_1|^{1/8} ))^{1/2} \Big]  
\le C_8  \rho^{1- \frac{2n}{p}} (\rho^{3/8} + \rho^{\frac{1}{2} (1- \frac{{\theta}}{8}) }). 
\end{gather*}
Since  $(1- \frac{2n}{p}) + {3/8} >1$  and  
 $(1- \frac{2n}{p})$ $+ \frac{1}{2} (1- \frac{{\theta}}{8}) >1$, we arrive at 
\begin{gather*}  
 {\Gamma_2} + {\Gamma_3} \le C_9 \rho^{\frac{5}{4}} =
  C_9 |s_3 -s_1|^{5/4}.
\end{gather*} 
Finally,  by the independence of  increments  {\!}    and using \eqref{vaii}, we find
$$
\begin{array}{l}
 {\Gamma_1}  \le 
 (\P ( |{\tilde M}_{s_3} - {\tilde M}_{s_2}| >    |s_3 -s_2|^{1/8})  
 \cdot  (\P ( |{\tilde M}_{s_2} - {\tilde M}_{s_1}| >    |s_2 -s_1|^{1/8})
\\  
\le  2c_3^2 \,   (\rho^{3/2} + \rho^{2(1- \frac{{\theta}}{8})}) 
 \le  c_4 |s_3- s_1|^{3/2}.
\end{array}
 $$  
Collecting 
 the previous  {\!}    estimates  {\!}    we finish the proof  as  {\!}    in \cite{Pr} obtaining
$$
 \begin{array}{l}
\E \big [ \big ( d_0  (Y^{s_1} , Y^{{s_2}} )  \cdot 
  d_0 ( Y^{s_2} ,  Y^{{s_3}} ) \big)^{8/\beta} \big]
 \; \le \,  C_{0} |s_3- s_1|^{5/4}.   
\end{array}
$$
 \end{proof}
  By  Theorem  \ref{modi12} 
 and using  the projections  {\!}    ${\pi}_x$ (cf. \eqref{pro2}),
  we write, for {\!}     $s  {\!}    \in [0,T]$, $x \in {\R}^n$,  
 $$
Z^s  {\!}    = (Z^{s,x})_{x \in {\R}^n}, \;\;\; \text{with} \;  
\pi_x (Z^s)= Z^{s,x} \in G_0. 
$$   
Recall that on some $\Omega_s$ (almost sure event)  $Y^{s,x} 
 = Z^{s,x}$, $s  {\!}    \in [0,T],    x \in \R^n$
(cf. \eqref{cte}).  
     The next result can be proved in the same way as  {\!}    Lemma {\!} 4.5 in \cite{Pr} replacing $b, d$ and $L$ with $\tilde b, n$ and $\tilde M$ respectively. 
 \begin{lemma}
\label{le1} Consider {\!}     the c\`adl\`ag process  {\!}     $Z$ which takes  {\!}     values  {\!}    in $C(\R^n; G_0)$ (see Theorem \ref{modi12}). The following assertions  {\!}    hold:
\hh (i) There exists  {\!}     $\Omega_1$
 (an almost sure event independent of $s,t$ and $x$)
such that for {\!}     any $\omega \in {\Omega}_1$, we have
 that $t \mapsto {\!} \tilde M_t (\omega)$ is  {\!}    c\`adl\`ag, $\tilde  M_0(\omega)=0$ and 
 $s  {\!}    \mapsto {\!} Z^s(\omega)$ is  {\!}    c\`adl\`ag; further, for {\!}     any $\omega \in \Omega_1$, 
$$
  \;\; Z^{s,x}_t(\omega) = x + \int_s^{t} \tilde b(r, Z^{s,x}_r(\omega)  + \tilde M_{r}(\omega) - \tilde M_s(\omega))dr, \; 0 \le s  {\!}    \le t \le T, \,x \in {\R}^n.
$$
Moreover,  for {\!}     $s  {\!}    \le t $, the r.v. $Z^{s,x}_t$ is  {\!}    ${\cal F}_{s,t}^L$-measurable (if $t \le s  {\!}    $, $Z^{s,x}_t =x$).

\hh (ii) There exists  {\!}    an almost sure event ${{\Omega}_2}$  and  a {\!}  ${\cal B}([0,T]) \times  {\cal F}$-measurable function  $V_{m}: [0,T] \times  \Omega $ $\to {\!} [0, \infty]$, with ${\int_0^T V_{m}(s,\omega)ds  {\!}    < \infty}$, for {\!}     any integer {\!}     $m >2n$, $\omega \in {\Omega}_2$, and, further, the following inequality holds  {\!}    on ${\Omega}_2$ 
\begin{gather*} 
 \sup_{t \in [0,T]} |Z_t^{s,x} - Z_t^{s,y}|  
\le  |x-y|^{\frac{m - 2n}{m}} [(|x| \vee |y|)^{\frac{2n + 1}{m}} \vee 1] \, V_{m}(s, \cdot),  \; x,y \in {\R}^n, \, s  {\!}    \in [0,T].
\end{gather*} 
(iii) There exists  {\!}    {\ an almost sure event}  ${\Omega}_3$ such 
that for {\!}     any $\omega \in \Omega_3$ we have
\begin{equation}
 \label{flo2}
 Z_t^{s,x} (\omega) + L_u(\omega) - L_s(\omega)
  = Z_t^{u, \, Z_u^{s,x}(\omega) + L_u(\omega) - L_s(\omega)}\, (\omega), \,
\end{equation} 
for {\!}     any $x \in {\R}^n$, $s,u,t \in [0,T]$, with 
$0 \le s  {\!}    < u \le T$.
\end{lemma} 
\noindent  
{\bf  Proof of Theorem \ref{h322} .}  The  proof follows  {\!}    the   same lines  {\!}    of the one of Theorem 5.1 in \cite{Pr}, using the previous  {\!}    lemmas,  replacing $b, d$ and $L$ with $\tilde b, n$ and $\tilde M$ respectively.

\section{An example of  degenerate SDE } 

Let us  {\!}    consider
 \begin{equation}
\label{eq-SDE3}
\begin{cases}
d X_t   = V_t d t,\; \; \;\;\; X_0 =x \in \R^d
\\
d V_t   = F \left(  X_t,V_t \right)  d t + d W_{t}, \quad V_0  =v \in \R^d,\;  \; t \in [0,T].
\end{cases}
\end{equation}
where $W = (W_t)$    is  {\!}    a {\!} standard Wiener {\!}     process  {\!}    with values  {\!}    in ${\R}^{d}$ defined on  $(\Omega,{{\mathcal F}}, \P)$.  
  One can write   equation \eqref{eq-SDE3} in the form \eqref{d} with $n = 2d$ and $L = W$ by defining  $C \in \R^{2d} \otimes    \R^d$,  $A \in \R^{2d} \otimes    \R^{2d}$ and the drift $b: \R^{2d} \to {\!} \R^{2d}$ as  {\!}    follows  {\!}     
 \begin{equation} \label{nota}
C = \begin{pmatrix}    0  \\  I    \end{pmatrix},\;\;\; A = \begin{pmatrix}    	0 &   {I} \\	0 & 0 
       \end{pmatrix}, \;\; b(x,v) = \begin{pmatrix}    0  \\  F(x,v)    \end{pmatrix},\;\; (x,v) \in \R^{2d};
\end{equation} 
 here $I$ denotes  {\!}    the $d \times  d$ identity  matrix.   
  First we assume:  

  \vskip 1mm  \noindent 
{\bf (H)} $F : \R^{2d} \to {\!} \R^{d}$ is  {\!}    a {\!} bounded function and there exist $\beta' \in (0,1)$, $\gamma {\!} \in (2/3,1)$ and $C>0$ such that 
\begin{equation} \label{d2}
|F(x,v) - F(x',v')| \le C  (|x-x'|^{\gamma} + |v-v'|^{\beta'}),\;\;\;
 (x,v),\, (x',v') \in \R^{2d}.     
\end{equation}
 Note that {\bf (H)} implies  {\!}     assumption (H1) of \cite{CdR}  for {\!}     equation \eqref{eq-SDE3} (see also {\!} \cite{WaZa2} for {\!}     more general assumptions  {\!}    on \eqref{eq-SDE3}).    Thus  {\!}    under {\!}     {\bf (H)}     strong existence  and uniqueness  {\!}    hold on each $[0,T]$ by Theorem 1.1 in \cite{CdR}.  
 
 Adapting  the argument  of Section 1.6 in \cite{CdR} 
 from the case $p=2$ to {\!} the case $p >2$ 
 (or {\!}     applying formula {\!} (1.19) of Theorem 1.7 of \cite{WaZa2}) we obtain the following result. 
 \begin{proposition}\label{pr} 
  Let us  {\!}    assume {\bf (H)}.  
 Let $(Z_t^{(x,v)})$ be the unique  strong solution to {\!} \eqref{eq-SDE3} starting from  $(x,v) \in \R^{2d}$ at time $s=0$. Then, for {\!}     any $T>0$, $p \ge 2$, there exists  {\!}    $C_{T,p} >0 $ such that 
 \begin{equation*}  
\E\Big[ \sup_{t \in [0,T]}\big| Z_t^{(x,v)} - Z_t^{(x',v')} \big|^p \Big] \le C_{T,p}  ( |x-x'|^p + |v-v'|^p), \;\; (x,v),  (x',v') \in \R^{2d}. 
\end{equation*}
 \end{proposition}
 We can prove Davie's  {\!}    uniqueness  {\!}    for {\!}     \eqref{eq-SDE3} applying Corollaries  {\!}     \ref{unb}  and \ref{cons}.
 
\begin{theorem} Let us  {\!}    consider {\!}     SDE \eqref{eq-SDE3} where  $W = (W_t)$    is  {\!}    a {\!} standard $\R^d$-valued Wiener {\!}     process  {\!}     defined on a {\!} probability space $(\Omega,{{\mathcal F}}, \P)$.

(i) Assume that $F : \R^{2d} \to {\!} \R^{d}$ is  {\!}    continuous  {\!}    and has  {\!}    at most a {\!} linear {\!}     growth (i.e., there exists  {\!}    $c>0$ such that $|F(x,v)| \le c (1 + |x|   + |v|)$,  $x, v \in \R^d$). 

(ii) Assume that there exist $\beta' \in (0,1)$ and  $\gamma {\!} \in (2/3,1)$ such that  for {\!}     any   $\eta \in C_0^{\infty}(\R^{2d})$ the function $\eta \cdot F : \R^{2d} \to {\!} \R^d$ verifies  {\!}    {\bf (H)} when $F$ is  {\!}    replaced by $\eta \cdot F$. 

\smallskip 
Then,    
 there exists  {\!}    an almost sure  event $\Omega' \in {\mathcal F}$  such that for {\!}     $\omega \in \Omega'$, $(x,v) \in {{\R}}^{2d}$,  the following integral equation in the unknown function ${z(t) = (x(t), v(t)) \in \R^{2d}}$      
\begin{gather*}  
\begin{cases}
x(t) =  x + t v 
+ \int_0^t (t-s) F(x(s), v(s)) \, ds  {\!}     + \int_0^t  W_s(\omega) \, d s  {\!}    
\\ 
v(t) = v + \int_0^t F(x(s), v(s)) \,  ds  {\!}    + W_t(\omega),   
\end{cases}     
\end{gather*}
 has  {\!}    {exactly} one solution $z(t)$ in $C([0,T];$ ${{\R}}^{2d})$. 
\end{theorem}
\begin{proof}
 With the notations  {\!}    in \eqref{nota} one can    check all the assumptions  {\!}    of  Corollary \ref{unb}  about SDE \eqref{eq-SDE3}. 
  Indeed  hypothesis  {\!}    (i) of Corollary \ref{cons} holds  {\!}    with $\beta = \beta' \wedge \gamma$. The integrability condition (ii) is  {\!}    clearly satisfied by the Wiener {\!}     process  {\!}    $W$. 
 Let us  {\!}    check   condition (iii). 
   
 For {\!}     any $\rho {\!} \in C_0^{\infty}(\R^{2d})$, we know that 
 $F \cdot \rho$ verifies  {\!}    {(\bf H)}. By Proposition \ref{pr} and Lemma {\!} \ref{basic} 
  we find that the SDE
 $$
 \begin{cases}
d X_t   = V_t d t,\; \; \;\;\; X_0 =x \in \R^d
\\
d V_t   = (F\cdot \rho) \left(  X_t,V_t \right)  d t + d W_{t}, \quad V_0  =v \in \R^d.
\end{cases}
 $$
 verifies  {\!}    Hypothesis  {\!}    \ref{qq}. This  {\!}    shows  {\!}    that condition (iii) holds. 
  By Corollaries  {\!}    \ref{unb} and \ref{cons}  we obtain  the assertion.
\end{proof}

\begin{remark} \label{forse} {\em
 One  could  write \eqref{eq-SDE3} as
 $
 dZ_t =   \begin{pmatrix}    V_t  \\ F(Z_t)    \end{pmatrix} dt + dL_t
 $
 with $Z_t = \begin{pmatrix}    X_t  \\  V_t    \end{pmatrix}$    and $L_t =\begin{pmatrix}    0  \\  W_t    \end{pmatrix} $ in order {\!}      to {\!} try to {\!} apply directly the results  {\!}    in \cite{Pr} to {\!} get Davie's  {\!}    uniqueness. However, a {\!} difficulty appears.  Assume that  $F$ verifies  {\!}     {\bf (H)}.  Since 
 the drift $b (x,v)=  \begin{pmatrix}    v  \\ F(x,v)    \end{pmatrix} $ is  {\!}    not bounded one
   should  truncate such  drift and localize according to {\!} 
 Corollary 5.4 in \cite{Pr}. A possible strategy 
         would be to {\!} look for {\!}     
 approximating  bounded    drifts  {\!}     like 
 $
 b_n(x,v) 
  =  \begin{pmatrix}     \eta_n(v)  \\ F(x,v)    \end{pmatrix},$ $ (x,v) \in {\R}^{2d},
 $  $n \ge 1.$  
 However {\!}     since $\eta_n$ is  {\!}    bounded it  cannot  satisfy  assumption (H3-b) in \cite{CdR}; this  {\!}    hypothesis  {\!}      is  {\!}    needed  to {\!} prove  strong uniqueness  {\!}    for {\!}     the approximating SDE $dZ_t^n = b_n(Z^n_t)dt  + dL_t$. 
  }
\end{remark}

\begin{remark} {\em One      can   obtain  Davie's  {\!}    type uniqueness  {\!}    results  {\!}     for {\!}     degenerate SDEs  {\!}    more general  than \eqref{eq-SDE3}, starting from known pathwise uniqueness  {\!}    results  {\!}    available in the literature (cf. \cite{CdR}, \cite{WaZa2}, \cite{men} and see the references  {\!}    therein).  
 For {\!}     instance, one could     consider {\!}     SDEs  {\!}    in $\R^{3d}$ like  
  \begin{equation} \label{dee}
\begin{cases}
dX_t =  F(X_t, Y_t, Z_t) dt + dW_t,
\\
 d Y_t   = X_t dt +  G(Y_t, Z_t) d t,   
 \\ 
 d Z_t = Y_t  d t + H(Z_t)dt.    
\;\;\;  X_0    =x \in \R^d, \;\; Y_0  = y\in \R^d, \; Z_0  =z \in \R^d.
\end{cases}
\end{equation}
 Such equations  {\!}    are a {\!} special case  of singular {\!}     degenerate SDEs  {\!}    considered in \cite{men}.
In \cite{men} there are conditions  {\!}    on $F,G$ and $H$ such that strong uniqueness  {\!}    holds  {\!}    for {\!}     \eqref{dee}.   
}
\end{remark}
 
\def\ciaoo{


 
%






\authorrunninghead{Metafune, Pallara, Priola}
\titlerunninghead{Ornstein-Uhlenbeck operators}

\title{Spectrum of Ornstein-Uhlenbeck operators
in $L^p$ spaces  {\!}    with respect to {\!} invariant measures}

\author{G. Metafune, D. Pallara}
\affil{Dipartimento {\!} di {\!}  Matematica {\!} ``Ennio {\!} De
Giorgi'', Universit\`a {\!} di {\!}  Lecce, C.P.193, 73100, Lecce, Italy}
\email{e-mail: metafune@le.infn.it, pallara@le.infn.it}
\and
\author{E. Priola\thanks{Partially supported by  the Italian National
Project  MURST ``Analisi {\!}  e controllo {\!} di {\!}  equazioni {\!}  di {\!}  evoluzione
deterministiche e stocastiche''}}
\affil{Dipartimento {\!} di {\!}  Matematica, via {\!} Carlo {\!} Alberto, 10, 10123 Torino, Italy}
\email{priola@dm.unito.it}

\abstract{Let $A=\sum_{i,j=1}^N q_{ij}D_{ij}+\sum_{i,j=1}^{N}b_{ij}x_j D_i$
be a {\!} possibly degenerate Orn\-stein-Uhlenbeck operator {\!}     in $\re^N$ and
assume that the associated Markov semigroup has  {\!}    an invariant measure $\mu$.
We compute the spectrum of $A$ in $L^p_\mu$ for {\!}     $1\leq p<\infty$.}


\begin{article}

\section{Introduction}

In this  {\!}    paper {\!}     we study the spectrum of the Ornstein-Uhlenbeck
operator
\begin{equation}               \label{OU}
A=\sum_{i,j=1}^N q_{ij}D_{ij}+\sum_{i,j=1}^{N}b_{ij}x_j D_i={\rm
Tr}(QD^2)+\langle Bx,D \rangle ,\ \ x\in \re^N,
\end{equation}
where $Q=(q_{ij})$ is  {\!}    a {\!} real, symmetric and nonnegative matrix and
$B=(b_{ij})$ is  {\!}    a {\!} non-zero {\!} real matrix.  The associated Markov semigroup
$\Tt$ has  {\!}    the following explicit representation, due to {\!} Kolmogorov
\begin{equation}               \label{OUS}
(T(t)f)(x)=\frac{1}{(4\pi)^{N/2}(
\det Q_t)^{1/2}}\int_{\re^N }e^{-<Q_t^{-1}y,y>/4} f(e^{tB}x-y)\, dy,
\end{equation}
where
$$
Q_t=\int_0^te^{sB}Qe^{sB^*}\, ds
$$
and $B^*$ denotes  {\!}    the adjoint matrix of $B$, see for {\!}     instance
\cite{DPL}. We assume that the spectrum
of $B$ is  {\!}    contained in  $\cc^- =\{ \lambda {\!} \in \cc:\; $Re
$(\lambda) <0 \}$.  Moreover {\!}     we  require that $\det Q_t >0$ for
any $t>0$ (that is, $Q_t$ is  {\!}    positive
definite); this  {\!}    is  {\!}    clearly true, in particular, if $Q$ is  {\!}    invertible.
We point out that the condition $\det Q_t >0$, $t>0$, is
equivalent to {\!} the hypoellipticity of $A$, see
\cite{Hormander}, and it can be also {\!} expressed by saying
that the kernel of $Q$ does  {\!}    not contain any invariant subspaces  {\!}    
of $B^*$ (see \cite{Hormander,LanPol,Lunardi2,priola1}).

Assuming that $\det(Q_t)>0$, in \cite[section 11.2.3]{DPZ2} it
is  {\!}    proved that  $\sigma(B)\subset \cc^-$  is  {\!}    equivalent to
the existence of an invariant measure $\mu$ for {\!}     $T_t$, i.e.
a {\!} probability measure on $\re^N$ such that
$$
\int_{\re^N}\bigl(T(t)f\bigr)(x)\, d\mu (x)
=\int_{\re^N} f(x)\, d\mu (x),
$$
for {\!}     every $t\geq 0$ and $f\in C_b(\re^N)$, the space of all continuous
and bounded functions  {\!}    on $\re^N$. Moreover {\!}      the invariant measure
$\mu$ is  {\!}    unique and it is  {\!}    given by $d\mu(x)=b(x)\, dx$ where
\begin{equation}                \label{defb}
b(x)=\frac{1}{(4\pi)^{N/2}(\det Q_\infty)^{1/2}}e^{-<Q_\infty^{-1}x,x>/4}
\end{equation}
and
$$
Q_\infty=\int_0^\infty e^{sB}Qe^{sB^*}\, ds.
$$
For {\!}     more information on invariant measures  {\!}    we refer {\!}     to
\cite{DPZ,KryRockZa}.  It is  {\!}    well known that
$\Tt$ extends  {\!}    to {\!} a {\!} strongly continuous  {\!}    semigroup of positive
contractions  {\!}    in $L^p_\mu=L^p(\re^N,d\mu)$ for {\!}     every
$1\leq p <\infty$.  Remark that, since $Q_t <Q_\infty$
in the sense of quadratic forms,  the integral in (\ref{OUS})
converges  {\!}    for {\!}     every $f\in L^p_\mu$ and $x \in \re^N$, so {\!} that
the extension of $\Tt$ to {\!} $L^p_\mu$ is  {\!}    still given by (\ref{OUS}).

Let us  {\!}    denote by $(A_p,D_p)$ the generator {\!}     of $\Tt$ in $L^p_\mu$.
The main aim of this  {\!}    paper {\!}     is  {\!}    the computation of the spectrum
of $(A_p,D_p)$ for {\!}     $1\leq p<\infty$.  If $1<p<\infty$, it is
known that the spectrum is  {\!}    discrete and consists  {\!}    of
eigenvalues  {\!}    of finite multiplicities, since the resolvent
is  {\!}    compact, see \cite{ChoGol}. We first prove  that all
the eigenfunctions  {\!}    are polynomials  {\!}    and then we arrive at
a {\!} complete characterization of the spectrum, see Section 3.
Our {\!}     method  shows  {\!}    that  it is  {\!}    possible  to {\!} reduce the computation
of the spectrum of $A$ to {\!} that of its  {\!}    drift term
$\langle Bx,D \rangle $, no {\!} matter {\!}     what the diffusion term
${\rm Tr}(QD^2)$ is, see in particular {\!}     Lemma {\!} \ref{riduzione}.

As  {\!}    a {\!} by-product of our {\!}     proof, we also {\!} show that the spectrum
is  {\!}    independent of $p \in ]1,\infty[$ (the $p$-independence of
the spectrum is  {\!}    however {\!}     a {\!} consequence of the compactness  {\!}    of the
resolvent, see e.g. \cite{Arendt}).  For {\!}     $p=1$ we obtain that
the spectrum is  {\!}    completely different, see Section 4.
The spectrum in $L^1_\mu$ is  {\!}    the closed left half-plane and
moreover {\!}     every complex number {\!}     with negative real part
is  {\!}    an eigenvalue. Let us  {\!}    stress  {\!}     that we allow $Q$ to {\!}  have
rank strictly less  {\!}    than $N$; however {\!}      our {\!}      main result seems
to {\!} be new even in the non-degenerate case, that is  {\!}    when $Q$
is  {\!}    positive definite.

Let us  {\!}    mention another {\!}     result of the paper.
Assuming that $A$ is  {\!}    nondegenerate in \cite{Fuh} it is
shown that $T_t$ is  {\!}    analytic in $L^p_{\mu}$ even in the
infinite dimensional setting, $1<p<\infty $ (see also
\cite{DPL,Lunardi, Goldys}). Under {\!}     our {\!}     assumptions, in
Section 2 we show that the semigroup $T_t$  is  {\!}    differentiable in
$L^p_\mu$, for {\!}     $1<p<\infty$;  obviously, it is  {\!}    not so {\!} in
$L^1_\mu$ (see also {\!} Corollary \ref{corL1}).

We remark that in the particular {\!}     case $Q=I$, $B=-I$, it is  {\!}    known
that the spectrum in $L^2_\mu$ consists  {\!}    of the negative
integers  {\!}    and that the Hermite polynomials  {\!}    form a {\!} complete
system of eigenfunctions.  Moreover, the operator {\!}     $-A_2$
on $L^2_\mu$ is  {\!}    unitarily equivalent to {\!} a {\!} Schr\"odinger
operator {\!}     $-\Delta {\!} +V$ on $L^2(\re^N,dx)$, where $V$ is  {\!}    a {\!} quadratic
potential (see \cite{Meyer,Bakry}).
Finally we refer {\!}     to {\!} \cite{Metafune} for {\!}     the spectrum of $A$
in $L^p(\re^N,dx)$ and in spaces  {\!}    of continuous  {\!}    functions.

\bigskip\noindent
{\bf Notation.} If $C$ is  {\!}    a {\!} linear {\!}     operator, we denote by $\sigma(C)$,
$P\sigma(C)$ and $\rho(C)$, the spectrum, the point-spectrum and the
resolvent set of $C$, respectively.  The spectral bound $s(C)$ is  {\!}    defined
by $s(C)=\sup \{{\rm Re} \lambda {\!} :  \lambda {\!} \in \sigma  (C) \}$.
$C_b(\re^N)$ stands  {\!}    for {\!}     the Banach space of all real
continuous  {\!}    and bounded functions  {\!}    on
$\re^N$.  $C_0(\re^N)$ is  {\!}    the closed
subspace of $C_b(\re^N)$ of functions  {\!}    vanishing at infinity, $C_0^\infty
(\re^N)$ is  {\!}    the space of $C^\infty$-functions  {\!}    with compact support and
${\cal S}(\re^N)$ is  {\!}    the Schwartz class.  ${\cal P}_n$ is  {\!}    the space of all
polynomials  {\!}    of degree less  {\!}    than or {\!}     equal to {\!} $n$.  For {\!}     $1\leq p<\infty$ and
$k\in \nat$, $W^{k,p}(\re^N)$ are the usual Sobolev spaces, and we define
\begin{equation}                             \label{eq1}
W^{k,p}_\mu =\{u\in W^{k,p}_{\rm loc}(\re^N): D^\alpha {\!} u \in L^p_\mu \
{\rm for\ |\alpha|} \leq k \}.
\end{equation}
The norm in $L^p_\mu$ will be denoted by $\|\cdot \|_p$.  Sometimes  {\!}    we
write $A_p$ for {\!}     $(A_p,D_p)$.  Throughout this  {\!}    paper {\!}     $\nat$ indicates  {\!}    
the set of nonnegative integers  {\!}    and $\cc^-$, $\cc^+$ the open left and right
half-planes, respectively.

\section{Properties  {\!}    of $\Tt$}
In this  {\!}    section we collect some properties  {\!}    of $\Tt$ and of its  {\!}    generator
$(A_p,D_p)$ needed in the sequel.

We observe that $C_0^\infty (\re^N)$ is  {\!}    dense in $W^{k,p}_\mu$, $1\leq
p<\infty$ .  Indeed, a {\!} simple truncation argument shows  {\!}    that the set of
$W^{k,p}_\mu$-functions  {\!}    with compact support is  {\!}    dense and, given $u\in
W^{k,p}_\mu$ with compact support, the usual approximating functions
$\phi_\eps  {\!}    * u$ converge to {\!} $u$, as  {\!}    $\eps  {\!}    \to {\!} 0$, in $W^{k,p}(\re^N)$ and
hence in $W^{k,p}_\mu$.

As  {\!}    regards  {\!}    the domains  {\!}    $D_p$, we remark that $D_p \subset D_q$ if $p \geq
q$ and $A_pu=A_qu$ for {\!}     $u\in D_p$.  If $Q$ is  {\!}    non-degenerate, the domain
$D_2$ is  {\!}    nothing but the weighted Sobolev space $W^{2,2}_\mu$ and $A_2u=
Au$ for {\!}     $u\in D_2$ (see \cite{Lunardi}).  A similar {\!}     result seems  {\!}    not to {\!} be
known in the general case when $p\neq 2$.  However, $D_p=W^{2,p}_\mu$ if
$(A_2,D_2)$ is  {\!}    self-adjoint; this  {\!}    fact turns  {\!}    out to {\!} be equivalent to {\!} the
identity $BQ=QB^*$ and implies  {\!}    $Q$ positive definite
(see \cite{ChoGol2,ChoGol3}).

For {\!}     our {\!}     purposes, we only need the following simple lemma.
\begin{lemma}                  \label{dominio}
Let $1\leq p <\infty$.  If $u\in C^\infty(\re^N)$ is  {\!}    such that $D_{ij}u \in
L^p_\mu$ for {\!}     $i,j=1,\dots,N$ and $|x||D u|\in L^p_\mu$, then $u\in D_p$ and
$A_pu=Au$.  Moreover, the Schwartz class  {\!}    ${\cal S}(\re^N)$ is  {\!}    a {\!} core for
$(A_p,D_p)$.
\end{lemma}
\begin{proof}Observe that $Au\in L^p_\mu$.  Let $0\leq \phi {\!}  \in C_0^\infty
(\re^N)$ be such that $\phi(x)=1$ if $|x| \leq 1$ and define $u_n(x)=\phi
(x/n)u(x)$.  It is  {\!}    easily seen, using dominated convergence, that $u_n \to
u$ and $Au_n \to {\!} Au$ in $L^p_\mu$.  Since $u_n \in C_0^\infty(\re^N)$, it
is  {\!}    elementary to {\!} check that $(T(t)u_n-u_n)/t \to {\!} Au_n$ uniformly (hence in
$L^p_\mu$) as  {\!}    $t \to {\!} 0$.  Therefore, $u_n \in D_p$ and the equality
$Au_n=A_pu_n$ holds.  Letting $n\to {\!} \infty$ we obtain that $u\in D_p$ and
that $A_pu=Au$, since $(A_p,D_p)$ is  {\!}    closed.  Finally, since ${\cal
S}(\re^N)$ is  {\!}    contained in $ D_p$ and is  {\!}    $T(t)$-invariant, it is  {\!}    a {\!} core for
$(A_p,D_p)$.
\end{proof}

We discuss  {\!}    now some smoothing properties  {\!}    of $\Tt$, depending upon the
hypoellipticity condition $\det Q_t>0$. To {\!} this  {\!}    purpose, it is  {\!}    useful
to {\!} recall that the above condition is  {\!}    also {\!} equivalent to
the well-known Kalman rank condition
$$
{\rm rank \ } \Bigl [Q^{1/2},BQ^{1/2}, \dots, B^{N-1}Q^{1/2} \Bigr]=N,
$$
arising in control theory (see e.g.  \cite{Za}).  In the above formula, the
$N\times  N^2$ matrix in the left-hand-side is  {\!}    obtained by writing
consecutively the columns  {\!}    of the matrices  {\!}    $B^iQ^{1/2}$.  Moreover, 
if $0\leq m \leq N-1$ is  {\!}    the smallest integer {\!}     such that 
${\rm rank \ } \Bigl[Q^{1/2},BQ^{1/2}, \dots, B^{m}Q^{1/2} \Bigr]=N$, 
then
\begin{equation}                \label{eq0}
\|Q_t^{-1/2}e^{tB}\| \leq \frac{C}{t^{1/2 +m}},\qquad t \in (0,1]
\end{equation}
(see \cite{Seid}). Of course $m=0$ if and only if $Q$ is  {\!}    invertible.

The following lemma {\!} is  {\!}    a {\!} slight modification of a {\!} result proved, in
the infinite-dimensional setting, in \cite[Lemma {\!} 3]{ChoGol}. We give the
proof for {\!}     completeness.  The number {\!}     $m$ which appears  {\!}    in the statement is
that defined above.

\begin{lemma}                      \label{regularity}
Let $1<p<\infty$.  For {\!}     every $t>0$, $T(t)$ maps  {\!}    $L^p_\mu$ into
$C^\infty(\re^N) \cap W^{k,p}_\mu$ for {\!}     every $k\in \nat$.  Moreover, there
exists  {\!}    $C=C(k,p) >0$ such that for {\!}     every $f\in L^p_\mu$ the inequality
$$
\|D^\alpha {\!} T(t)f\|_p \leq
\frac{C}{t^{|\alpha|(1/2+m)}}\|f\|_p, \qquad t \in (0,1)
$$
holds  {\!}    for {\!}     every multiindex $\alpha$ with $|\alpha|=k$.
\end{lemma}
\begin{proof} Let us  {\!}    fix $t>0$ and set
$$
b_t(x)=\frac{1}{(4\pi)^{N/2}(\det Q_t)^{1/2}}e^{-<Q_t^{-1}x,x>/4}.
$$
Since $Q_t <Q_\infty$, in the sense of quadratic forms, it is  {\!}    easily seen
that there exist $K, \eps  {\!}    >0$ (depending upon $t$) such that $b_t(x) \leq
Ke^{-\eps  {\!}    |x|^2}b(x)$, where $b$ (defined in (\ref{defb})) is  {\!}    the density
of $\mu$.  It follows  {\!}    that one can differentiate under {\!}     the integral sign in
(\ref{OUS}) for {\!}     every $f\in L^p_\mu$ thus  {\!}    obtaining
$$
\bigl(D T(t)f \bigr {\!}     )(x)=- \frac{1}{2}
\int_{\re^N}e^{tB^*}Q_t^{-1}yf\bigl(e^{tB}x-y\bigr {\!}     ) b_t(y)\, dy
$$
for {\!}     every $x\in \re^N$ and hence $T(t)f \in C^1(\re^N)$.  By H\"older
inequality and (\ref{eq0})
\begin{eqnarray*}
&&|(D_iT(t))f(x)|
\\
&\leq& \frac{1}{2}
\Bigl(\int_{\re^N}|\langle Q_t^{-1/2}e^{tB}e_i, Q_t^{-1/2} y
\rangle |^{p'}b_t(y)\, dy\Bigr)^{1/p'}\Bigl
((T(t)|f|^p)(x)\Bigr)^{1/p}
\\
&\leq& \frac{1}{2}
|Q_t^{-1/2}e^{tB}e_i|\Bigl (\int_{\re^N}|Q_t^{-1/2}y|^{p'}b_t(y)\, dy
\Bigr)^{1/p'}
\Bigl ((T(t)|f|^p)(x)\Bigr)^{1/p} \\
&\leq &
C_p t^{-1/2-m}\Bigl ((T(t)|f|^p)(x)\Bigr)^{1/p}
\end{eqnarray*}
and the thesis  {\!}    follows  {\!}    for {\!}     $k=1$ raising to {\!} the power {\!}     $p$ and integrating
the above inequality with respect to {\!} $\mu$.  The proof for {\!}     $k\geq 1$
proceeds  {\!}    as  {\!}    in \cite[Lemma {\!} 3.2]{Lunardi} using the equality
$DT(t)u=e^{tB^*}T(t)Du$, which holds  {\!}    for {\!}     every $u\in W^{1,p}_\mu$.  This
identity is  {\!}    easily verified in $C_0^\infty (\re^N)$ and extends  {\!}    to
$W^{1,p}_\mu$ by density.
\end{proof}

The compactness  {\!}    of $\Tt$ for {\!}     $p=2$ easily follows  {\!}    from the above lemma {\!} and
the compactness  {\!}    of the imbedding of $W^{1,2}_\mu$ into {\!} $L^2_\mu$, see
\cite{DPZ}. If $1<p<\infty$, the same holds  {\!}    by interpolation
(see \cite[Lemma {\!} 2]{ChoGol}).

If $Q$ is  {\!}    non degenerate, the analyticity of $\Tt$ in $L^2_\mu$
was  {\!}    proved in \cite {Fuh} (see also {\!} \cite{DPL,Lunardi}).
>From the Stein interpolation theorem it follows  {\!}    that $\Tt$
is  {\!}    analytic in $L^p_\mu$ for {\!}     $1<p<\infty$.  On the other {\!}     hand,
$\Tt$ is  {\!}    not analytic in $L^2_\mu$ (hence in $L^p_\mu$)
if $Q$ is  {\!}    degenerate, see \cite{Goldys}.  We show that in any
case $\Tt$ is  {\!}    differentiable in $L^p_\mu$, if $1<p<\infty $.
To {\!} prove this  {\!}    we need the following lemma {\!} which generalises
\cite[Lemma {\!} 2.1]{Lunardi}.

\begin{lemma}                \label{alessandra}
If $1<p<\infty$, for {\!}     every $h=1,\ldots,N$ the map $u \mapsto {\!} x_hu$
is  {\!}    bounded from $W^{1,p}_\mu$ to {\!} $L^p_\mu$.
\end{lemma}
\begin{proof}It suffices  {\!}    to {\!} show that there is  {\!}    a {\!} constant $K_p$ such that
for {\!}     every $u\in C_0^\infty(\re^N)$
\begin{equation}                  \label{ceq}
\int_{\re^N}|x_hu(x)|^p \,d\mu(x) \leq
K_p \int_{\re^N}\bigl(|u(x)|^p+|D u(x)|^p\bigr)\, d\mu (x).
\end{equation}
By a {\!} linear {\!}     change of variables  {\!}    we may assume that $Q_\infty$ is  {\!}    diagonal
with eigenvalues  {\!}    $\mu_1,\dots,\mu_N$ and hence that
$$
b(x)=\frac{1}{(4\pi)^{N/2}(\mu_1 \cdots  {\!}    \mu_N)^{1/2}}
\exp\Bigl\{-\sum_{i=1}^Nx_i^2/(4\mu_i)\Bigr\}.
$$
First case, assume $p\geq 2$.  If $u\in C_0^\infty(\re^N)$, then
one has  {\!}    with $C=2\max \{\mu_1, \dots,\mu_N\}$
\begin{eqnarray*}
&&\!\!\int_{\re^N}|x_hu(x)|^p\,d\mu(x)\leq -C\int_{\re^N}|u(x)|^p
|x_h|^{p-2}x_h\cdot D_hb(x)\, dx
\\
&=\!\!&C\int_{\re^N} \bigl (px_hu(x)|x_hu(x)|^{p-2}D_h u(x) +
(p-1)|x_h|^{p-2}|u(x)|^p \bigr {\!}     )\, d\mu(x)
\\
&\leq\!\!& C_1 \int_{\re^N}|x_h|^{p-2}|u(x)|^p\, d\mu(x)
\\
&&+C_2 \Bigl(\int_{\re^N}|x_hu(x)|^p\, d\mu(x)\Bigr)^{\frac{p-1}{p}}
\Bigl(\int_{\re^N}|D_hu(x)|^p\, d\mu(x) \Bigr {\!}     )^{\frac{1}{p}}
\\
&\leq\!\!& \eps  {\!}    \int_{\re^N}|x_hu(x)|^p\, d\mu(x)
+C_\eps\int_{\re^N}(|u(x)|^p+|D_hu(x)|^p)\, d\mu(x),
\end{eqnarray*}
for {\!}     every $\eps  {\!}    >0$, with a {\!} suitable $C_\eps$ (in the last line we have
used Young's  {\!}    inequality and the estimate $|x_h|^{p-2}\leq C_\eps  {\!}    +\eps
|x_h|^p$).  Choosing $\eps  {\!}    <1$ we deduce (\ref{ceq}).

Let us  {\!}    deal with the case $1<p<2$. We proceed as  {\!}    before but
we have to {\!} estimate in a {\!} different way the term
$$
\int_{\re^N}|x_h|^{p-2} |u(x)|^p\, d\mu(x).
$$
To {\!} simplify the notation, take $h=N$ and write $x'=(x_1,\ldots,x_{N-1})$,
$b(x)=b'(x')\frac{e^{-x_N^2/ 4 \mu_N}}{(4\pi\mu_N)^{1/2}} $,
$d\mu'=b'(x')dx'$,
$d\mu''=(4\pi\mu_N)^{-1/2}\exp\{-x_N^2/4\mu_N\} dx_N$, so {\!} that
\begin{eqnarray*}
&&\int_{\re^N}|x_N|^{p-2} |u(x)|^p\, d\mu(x)
\\
&=&\int_{\re^{N-1}}d\mu'(x')\int_\re |x_N|^{p-2}|u(x',x_N)|^p\,d\mu''(x_N)
\\
&=&\int_{\re^{N-1}}d\mu'(x')\int_{|x_N|\geq
1}|x_N|^{p-2}|u(x',x_N)|^p\,d\mu''(x_N)
\\
&&+\int_{\re^{N-1}}d\mu'(x')
\int_{-1}^1|x_N|^{p-2}|u(x',x_N)|^p\,d\mu''(x_N)
\\
&:=& J_1 + J_2
\end{eqnarray*}
Clearly, $J_1\leq\int_{\re^N}|u(x)|^p\, d\mu(x)$. Let us  {\!}    estimate
$J_2$. For {\!}     every $x'\in\re^{N-1}$ we have, by the Sobolev embedding
$W^{1,p}(-1,1)\hookrightarrow L^\infty(-1,1)$,
\begin{eqnarray*}
&&\int_{-1}^1|x_N|^{p-2}|u(x',x_N)|^p\,d\mu''(x_N)
\\
&\leq&
C \Bigl(\sup_{|x_N|\leq 1}|u(x',x_N)\Bigr)^p
\int_{-1}^1|x_N|^{p-2}\,dx_N
\\
&\leq& C_1\int_{-1}^1\bigl(|u(x',x_N)|^p+|D_Nu(x',x_N)|^p\bigr)\,dx_N
\\
&\leq& C_2\int_{\re}\bigl(|u(x',x_N)|^p+|D_Nu(x',x_N)|^p\bigr)\,d\mu''(x_N)
\end{eqnarray*}
whence, integrating on $\re^{N-1}$,
$$
J_2 \leq C_2 \int_{\re^N} (|u(x)|^p + |Du(x)|^p)\, d\mu(x),
$$
and this  {\!}    completes  {\!}    the proof.
\end{proof}

It follows, in particular, that the map $Lu=\langle Bx, Du \rangle$ is
bounded from $W^{2,p}_\mu$ into {\!} $L^p_\mu$ for {\!}     $1<p<\infty$.

\begin{proposition}                   \label{differentiability}
For {\!}     $1<p<\infty$ the semigroup $\Tt$ is  {\!}    differentiable in $L^p_\mu$.
\end{proposition}
\begin{proof} If $f \in {\cal
S}(\re^N)$ then $T(t)f \in {\cal S}(\re^N) \subset D_p$.  From Lemmas  {\!}    \ref
{alessandra}, \ref{regularity} it follows  {\!}    as  {\!}    in \cite[Proposition
3.3]{Lunardi} that
$$
\|A_pT(t)f\|_p=\|AT(t)f\|_p \leq \frac{C}{t^{2m+1}}\|f\|_p, \quad 0<t\leq
1,
$$
hence $A_pT(t)$ extends  {\!}    to {\!} a {\!} bounded operator {\!}     in $L^p_\mu$ and the thesis
follows.
\end{proof}

We shall see that the above result is  {\!}    false for {\!}     $p=1$, see Section 4.

\section{Spectrum in $L^p_\mu$ for {\!}     $1<p<\infty$}

In this  {\!}    section we assume that $1< p<\infty$.  The following estimate is
the main step to {\!} show that the eigenfunctions  {\!}    of $A_p$ are polynomials.

\begin{lemma}                       \label{stima}
Let $k\in \nat$ and $\eps  {\!}    >0$ be given, with $s(B)+\eps<0$. Then there
exists  {\!}    $C=C(k,\eps)$ such that for {\!}     every $u\in W^{k,p}_\mu$
\begin{equation}                           \label{eq2}
\sum_{|\alpha|=k}\|D^\alpha {\!} T(t)u\|_p \leq C
e^{tk(s(B)+\eps)}\sum_{|\alpha|=k}\|D^\alpha {\!} u\|_p,\;\;\; t \ge 0.
\end{equation}
\end{lemma}
\begin{proof} Let $C_1=C_1(\eps)$ be such that $\|e^{tB^*}\|\leq C_1
e^{t(s(B)+\eps)}$ and recall that $D T(t)u=e^{tB^*}T(t)D u$ for {\!}     every $u\in
W^{1,p}_\mu$.  Since $\Tt$ is  {\!}    contractive in $L^p_\mu$ the statement is
proved for {\!}     $k=1$ with $C=C_1$.  Suppose that the statement is  {\!}    true for {\!}     $k$
with a {\!} suitable constant $C_k$ and consider {\!}     $u\in W^{k+1,p}_\mu$.  Then, if
$|\alpha|=k$,
\begin{eqnarray*}
\|D D^\alpha {\!} T(t)u\|_p&=&\|D^\alpha {\!} D T(t)u\|_p=\|D^\alpha {\!} e^{tB^*}T(t)D
u\|_p
\\
&\leq&  C_1e^{t(s(B)+\eps)}\|D^\alpha {\!} T(t)D u\|_p
\\
&\leq& C_1C_k e^{t(k+1)(s(B)+\eps)}\|D D^\alpha {\!} u\|_p.
\end{eqnarray*}
\end{proof}

Observe that $\sigma(A_p) \subset \{\lambda {\!} \in \cc:  {\rm Re}\lambda {\!} \leq
0 \}$, since $\Tt$ is  {\!}    a {\!} semigroup of contractions  {\!}    in $L^p_\mu$ and that
$0\in\sigma(A)$.  Moreover, every eigenfunction corresponding to {\!} the
eigenvalue $0$ is  {\!}    constant (this  {\!}    holds  {\!}    also {\!} for {\!}     $p=1$).
In fact, if $u\in D_p$ and $A_pu=0$, then
$T(t)u=u$.  On the other {\!}     hand (see \cite[Theorem 4.2.1]{DPZ})
$$
T(t)u \to {\!} \int_{\re^N}u\, d\mu
$$
as  {\!}    $t\to {\!} \infty$ and therefore $u$ is  {\!}    constant.  We now show that all the
eigenfunctions  {\!}    are polynomials.

\begin{proposition}                           \label{polinomio}
Suppose that $u\in D_p$ satisfies  {\!}    $A_pu=\lambda {\!} u$. Then $u$ is  {\!}    a
polynomial.
\end{proposition}
\begin{proof}Since $T(t)u=e^{\lambda {\!} t}u$, from Lemma {\!} \ref{regularity}
we deduce that $u\in W^{k,p}_\mu \cap { C}^{\infty}(\re^N)$,
for {\!}     every $k$.  Clearly $D^\alpha {\!} T(t)u=e^{\lambda {\!} t}D^\alpha {\!} u$
for {\!}     every multiindex $\alpha$.  Given $\eps\in (0,|s(B)|)$,
from Lemma {\!} \ref{stima} it follows  {\!}    that
$$
e^{t\,{\rm Re}\lambda {\!} }\sum_{|\alpha|=k}\|D^\alpha {\!} u\|_p\leq C(k,\eps)
e^{tk(s(B)+\eps)}\sum_{|\alpha|=k}\|D^\alpha {\!} u\|_p
$$
and hence $D^\alpha {\!} u=0$ if $|\alpha| |s(B)| > |{\rm Re} \lambda|$.
It follows  {\!}    that $u$ is  {\!}    a {\!} polynomial of degree less  {\!}    than or {\!}     equal to
$\frac{ {\rm Re}  (\lambda)}{s(B)}$.  This  {\!}    concludes  {\!}    the proof.
\end{proof}

Let us  {\!}    denote by
$$
Lu=\langle Bx,Du \rangle
$$
the drift term in (\ref{OU}). We reduce the computation of the
spectrum of $A_p$ to {\!} that of $L$.

\begin{lemma}                                   \label{riduzione}
The following statements  {\!}    are equivalent.
\begin{itemize}
\item[{\rm (i)}] $\lambda {\!} \in \sigma(A_p)$.
\item[{\rm (ii)}] There exists  {\!}    a {\!} homogeneous  {\!}    polynomial $u \neq 0$ such
that
$Lu=\lambda {\!} u$.
\end{itemize}
\end{lemma}
\begin{proof} First we observe that $A_pu=Au$ if $u$ is  {\!}    a {\!} polynomial (see
Lemma {\!} \ref{dominio}) and that both $A$ and $L$ map ${\cal P}_n$ into
itself.  Moreover {\!}     $A=L$ on ${\cal P}_1$ and hence we may consider {\!}     only
polynomials  {\!}    of degree greater {\!}     than or {\!}     equal to {\!} $2$.

Suppose that (i) holds  {\!}    and let $u $ be a {\!} polynomial of degree $n \ge 2$
such that $A_pu=\lambda {\!} u$, that is  {\!}    $\lambda
u-\sum_{i,j}q_{ij}D_{ij}u-Lu=0$.  If $\lambda {\!} -L$ is  {\!}    bijective on ${\cal
P}_{n-2}$ we can find $v\in {\cal P}_{n-2}$ such that $\lambda
v-Lv=\sum_{i,j}q_{ij}D_{ij}u$ and hence $z=u-v \in {\cal P}_n$, satisfies
$\lambda {\!} z-Lz=0$ and $z \neq 0$.  If $\lambda {\!} -L$ is  {\!}    not bijective on
${\cal P}_{n-2}$ we consider {\!}     a {\!} function $z$ in its  {\!}    kernel.  In any case we
find $0\neq z\in {\cal P}_n$ such that $\lambda {\!} z-Lz=0$.  To {\!} find a
(nonzero) homogeneous  {\!}    polynomial $u$ such that $\lambda {\!} u-Lu=0$ it is
sufficient to {\!} observe that $L$ maps  {\!}    homogeneous  {\!}    polynomials  {\!}    into
homogeneous  {\!}    polynomials  {\!}    so {\!} that all homogeneous  {\!}    addends  {\!}    $u$ of $z$ satisfy
$\lambda {\!} u-Lu=0$.

Assume now that (ii) holds  {\!}    with $u$ homogeneous  {\!}    polynomial of
degree $n \geq 2$.  If $\lambda-A_p$ is  {\!}    not injective on
${\cal P}_{n-2}$ clearly (i) is  {\!}    true.
Otherwise we find $v \in {\cal P}_{n-2}$ such that $\lambda
v-Av=\sum_{i,j}q_{ij}D_{ij}u$ and then $0 \neq w=u+v \in {\cal P}_n$
satisfies  {\!}    $\lambda {\!} w-A_pw=0$.
\end{proof}

We study now the equation $\gamma {\!} u-Lu=0$ with $u$ polynomial, $\gamma {\!} \in
\cc$.  If $B=-I$ this  {\!}    is  {\!}    the well-known Euler {\!}     equation satisfied by all
regular {\!}     functions  {\!}    homogeneous  {\!}    of degree $(-\gamma)$.  If we require that
$u$ is  {\!}    a {\!} polynomial, we obtain $(-\gamma)\in\nat$, hence all negative
integers  {\!}    are eigenvalues  {\!}    of $L$ and, for {\!}     every $n\in\nat$, all homogeneous
polynomials  {\!}    of degree $n$ are eigenfunctions.

The equation with a {\!} general $B$ is  {\!}    much more complicated and we shall not
characterise all polynomial solutions  {\!}    but only the values  {\!}    of $\gamma {\!} $ for
which such a {\!} solution exists.  Observe that a {\!} differentiable function $u $
satisfies  {\!}    $\gamma {\!} u-Lu=0$ if and only if
\begin{equation}                                   \label{eq}
u(e^{tB}x)=e^{t\gamma}u(x) \qquad t\geq 0,\ x\in \re^N.
\end{equation}

Let $u$ be a {\!} (nonzero) homogeneous  {\!}    polynomial of degree $n$ satisfying
(\ref{eq}): in this  {\!}    case the same equality holds  {\!}    for {\!}     every {\em complex}
point $x\in\cc^N$.  Let now $M$ be a {\!} non-singular {\!}     complex $N\times  N$
matrix, such that $MBM^{-1}=C$, where $C$ is  {\!}    the canonical Jordan form of
$B$.  Introduce a {\!} new homogeneous  {\!}    polynomial $v(z)= u(M^{-1}z)$,
$z\in\cc^N$, so {\!} that $u(x)=v(Mx)$.  Since $v(Me^{tB}M^{-1}z)=e^{t\gamma}\,
v(z)$, we obtain that
$$
v(e^{tC} z)=e^{t \gamma}\, v(z),\qquad z \in \cc^N,
$$
and we find the values  {\!}    of $\gamma$ for {\!}     which a {\!} solution exists  {\!}    working with
the Jordan matrix $C$.  Before proving the main result of this  {\!}    section, we
present in a {\!} particular {\!}     case the argument we use in the proof.  Let us
suppose that $C$ consists  {\!}    of a {\!} unique Jordan block of size $N$ relative to
an eigenvalue $\lambda$, that is
$$
C=\left(
\begin{array}{cccc}
\lambda {\!} &       1 & \cdots  {\!}    & 0        \\
0       & \lambda {\!} & \cdots  {\!}    & \vdots  {\!}      \\
\vdots  {\!}     & \vdots  {\!}     & \ddots  {\!}    & 1        \\
0       & \cdots  {\!}     & 0      & \lambda {\!}  \\
\end{array}
\right)
$$
and write $C=\lambda {\!} I+R$ with $R$ nilpotent. Hence $e^{tR}$ has
polynomial entries  {\!}    and we obtain
\begin{equation} \label{23}
e^{t\gamma}v(z)=v(e^{tB}z) =v(e^{t\lambda}e^{tR}z)
=e^{n\lambda {\!} t}v(e^{tR}z)  =e^{n\lambda {\!} t}q(t,z)
\end{equation}
where $q(t,z)=\sum_{|\alpha|=n}c_{\alpha}(t)z^{\alpha}$ and the
$c_{\alpha}(t)$ are polynomials.  Now fix $\hat{z}\neq 0$ in (\ref{23})
such that $v(\hat{z})\neq 0$ and look at the variable $t$.  It follows  {\!}    that
$\gamma=n\lambda$, i.e., the eigenvalues  {\!}    of $L$ are multiples  {\!}    of the
(unique) eigenvalue of $B$. In the general case, we have the following
result.

\begin{theorem}                    \label{spettro}
Let $\lambda_1, \dots,\lambda_r$ be the (distinct) eigenvalues  {\!}    of $B$.
Then
$$
\sigma(A_p)=\Bigl\{\lambda=\sum_{j=1}^r {\!}     n_j \lambda_j : n_j\in \nat
\Bigr\}.
$$
\end{theorem}
\begin{proof}We keep the above notation (recall that $M$ is  {\!}    a {\!} non-singular
complex $N\times  N$ matrix, such that $MBM^{-1}=C$ and $C$ is  {\!}    the canonical
Jordan form of $B$). Let $C_j$, for {\!}     $j=1,\ldots  {\!}    r$, be the
Jordan block of $C$ associated with $\lambda_j$ and denote by
$k_j$ ($1\leq k_j \leq N$, $\sum_{j=1}^r {\!}     k_j=N$) the size of $C_j$.
We may write $C_j=\lambda_jI+R_j$ where  $R_j$ is  {\!}    a {\!} nilpotent matrix.
Let us  {\!}    decompose $\cc^N$ into {\!} the direct sum of the invariant
subspaces  {\!}    corresponding to {\!} the Jordan blocks  {\!}    of $C$ and write
$z\in\cc^N$ in the form $z=(z_1,\ldots,z_r)$, with $z_j\in \cc^{k_j}$.

Assume that $\gamma\in\sigma(A_p)$.  Then, according to {\!} Lemma
\ref{riduzione}, there exists  {\!}    a {\!} nonzero {\!} homogeneous  {\!}    polynomial $u$ such
that $Lu=\gamma {\!} u$ or, in an equivalent way, $u(e^{tB}x)= e^{\gamma
t}u(x)$.  Introducing the homogeneous  {\!}    polynomial $v(z)= u(M^{-1}z)$, we
know that $v(e^{tC}z)=e^{t\gamma}v(z)$ for {\!}     every $z\in\cc^N$.  Let us  {\!}    write
$v$ in the following way:
$$
v(z)=\sum_{|\alpha_1|+\ldots+|\alpha_r|=n}
c_{\alpha_1,\ldots,\alpha_r}\ \prod_{j=1}^rz_j^{\alpha_j},
$$
and prove that $\gamma=\sum_j\lambda_j|\alpha_j|$, for {\!}     suitable
$(\alpha_j)$.  We have
\begin{eqnarray*}
e^{t\gamma}v(z)&=&v(e^{tC}z)=v(e^{tC_1}z_1,\ldots,e^{tC_r}z_r)
\\
&=&\sum_{|\alpha_1|+\ldots+|\alpha_r|=n}
c_{\alpha_1,\ldots,\alpha_r}\ \prod_{j=1}^r(e^{tC_j}z_j)^{\alpha_j}
\\
&=&
\sum_{|\alpha_1|+\ldots+|\alpha_r|=n}
c_{\alpha_1,\ldots,\alpha_r}\
e^{t(\lambda_1|\alpha_1|+\ldots+\lambda_r|\alpha_r|)}
\ \prod_{j=1}^r(e^{tR_j}z_j)^{\alpha_j}.
\end{eqnarray*}
Now fix $\hat{z}\neq 0$ such that $v(\hat{z})\neq 0$ and look at the
variable $t$.  Since $\prod_{j=1}^r(e^{tR_j}\hat{z}_j)^{\alpha_j}$ is  {\!}    a
polynomial in $t$ for {\!}     any $(\alpha_1,\ldots,\alpha_r)$, it follows  {\!}    that
there exists  {\!}    some $(\alpha_1,\ldots,\alpha_r)$ such that
$\gamma=\lambda_1|\alpha_1|+\ldots+\lambda_r|\alpha_r|$.  This  {\!}    means  {\!}    that
\begin{equation} \label{ci}
\gamma {\!} =\sum_{j=1}^r {\!}     n_j \lambda_j,  \;\;\; n_j\in \nat.
\end{equation}
Conversely, let $\gamma {\!} =\sum_{j=1}^r {\!}     n_j \lambda_j$, with arbitrary
$n_j\in \nat$. Let us  {\!}    write $z\in\cc^N$ in the form
$$
z=(z_1,\ldots,z_r)=
(z_1,\ldots,z_{k_1},z_{k_1 +1},\ldots,z_{k_1+k_2},\ldots,
z_{k_1+\ldots+k_r}).
$$
Consider {\!}     the polynomial
$$
v(z)=z_{k_1}^{n_1} \cdot z_{k_1 + k_2}^{n_2} \cdots  {\!}    z_{k_1+ \ldots
k_r}^{n_r},
$$
depending only upon the $r$ complex variables
$z_{k_1},z_{k_1+k_2},\ldots,z_{k_1+\ldots  {\!}    k_r}$ (the last variable in 
each block). It is  {\!}    easy to {\!} verify that
$$
v(e^{tC}z)=e^{t\gamma}v(e^{tR_1}z_1,\ldots,e^{tR_r}z_r)=e^{t\gamma}v(z),
\quad z\in\cc^N.
$$
The polynomial $u(z)= v(Mz)$, $z \in \cc^N$, satisfies
$u(e^{tB}x)= e^{t\gamma}u(x)$, $x \in \re^N $. It follows  {\!}    that
$Lu=\gamma {\!} u$ and hence $\gamma\in\sigma_p(A)$, by Lemma {\!} \ref{riduzione}.
\end{proof}

\section{Spectrum in $L^1_\mu$}

We show that the spectrum of $A_1$ is  {\!}    the left half-plane.  In particular
$\Tt$ is  {\!}    not norm-continuous  {\!}    in $L^1_\mu$, hence not analytic, nor
differentiable, nor {\!}     compact (see \cite[Ch. II, Sec. 4]{EN}).

\begin{theorem}          \label{specL1}
The spectrum of $(A_1,D_1)$ is  {\!}    the left half-plane $\{\lambda {\!} \in \cc:
{\rm Re}\ \lambda {\!} \leq 0\}$.  Each complex number {\!}     $\lambda$ with ${\rm Re}\
\lambda {\!} <0$ is  {\!}    an eigenvalue.
\end{theorem}
\begin{proof}Let $b$ be the density of $\mu$ with respect to {\!} the Lebesgue
measure, given by (\ref{defb}), and set $h=1/b$.  Let $\Phi:
L^1=L^1(\re^N,dx) \to {\!} L^1_\mu$ be the isometry defined by
$$
(\Phi {\!}  u)(x)=u(x)h(x), \qquad u\in L^1, \quad x \in \re^N.
$$
We define an operator {\!}     $(G,D_G)$ on $L^1$ by $D_G=\Phi^{-1}(D_1)$ and
$G=\Phi^{-1}A_1\Phi$.  If $u\in C_0^\infty (\re^N)$, then $u\in D_G$ and
\begin{eqnarray*}
Gu(x)&=&b(x)\bigl(A(uh)\bigr)(x)
\\
&=&Au(x)+2b(x)\sum_{i,j=1}^N q_{ij}D_ih(x) D_ju(x) + b(x)u(x) Ah(x).
\end{eqnarray*}
A direct computation shows  {\!}    that
$$
2b(x)\sum_{i,j=1}^N q_{ij}D_ih(x) D_ju(x)=
\langle QQ_{\infty}^{-1} x , Du(x)\rangle
$$
and
\begin{eqnarray*}
b(x)Ah(x)&=&\Bigl[ \frac{1}{2} {\rm Tr}(Q Q_{\infty}^{-1} ) +
\frac{1}{4}  \langle  Q Q_{\infty}^{-1} x,  Q_{\infty}^{-1} x \rangle
+\frac{1}{2} \langle B^*  Q_{\infty}^{-1} x,x \rangle\Bigr]
\\
&=& \Bigl[ \frac{1}{2} {\rm Tr}(Q Q_{\infty}^{-1} ) +
\frac{1}{4} \langle  Q Q_{\infty}^{-1} x,  Q_{\infty}^{-1} x \rangle
+ \frac{1}{2} \langle B Q_{\infty}   Q_{\infty}^{-1} x, Q_{\infty}^{-1} x
\rangle\Bigr].
\end{eqnarray*}
Using the identity $BQ_\infty +Q_\infty B^*=-Q$, which implies
$2\langle BQ_\infty x,x\rangle=-\langle Qx,x\rangle$,
it follows  {\!}    that
$\frac{1}{4}\langle Q Q_{\infty}^{-1} x, Q_{\infty}^{-1} x \rangle
+\frac{1}{2} \langle B Q_{\infty}   Q_{\infty}^{-1} x, Q_{\infty}^{-1} x
\rangle = 0$ and hence, setting $k=\frac{1}{2}{\rm Tr}(QQ_{\infty}^{-1})$,
\begin{eqnarray*}
Gu(x)&=& Au(x) + \langle QQ_{\infty}^{-1} x , Du(x)\rangle
         +ku(x)
\\
     &=&{\rm Tr}(QD^2u(x))+\langle(B+QQ_{\infty}^{-1}) x,Du(x)\rangle
         +ku(x)
\\
     &=&{\rm Tr}(QD^2u(x))-\langle(Q_\infty B^*Q_{\infty}^{-1})x,Du(x)\rangle
         +ku(x).
\end{eqnarray*}
The operator {\!}     $G_0={\rm Tr}(QD^2)-\langle(Q_\infty
B^*Q_{\infty}^{-1})x,D\rangle$, with a {\!} suitable domain $D_{G_0}$, is  {\!}    the
generator {\!}     of an Ornstein-Uhlenbeck semigroup in $L^1$.  Even
though an explicit description of $D_{G_0}$ is  {\!}    not known, we point out that
$ C_0^\infty (\re^N)$ is  {\!}    a {\!} core of $(G_0,D_{G_0})$ (see \cite[Proposition
3.2]{Metafune}).  The above computation shows  {\!}    that $G=G_0+kI$ on
$C_0^\infty (\re^N)$ and therefore $D_{G_0} \subset D_G$ and $G=G_0+kI$ on
$D_{G_0}$, since $(G,D_G)$ is  {\!}    closed.  On the other {\!}     hand, if $\lambda$ is
sufficiently large, $\lambda-G$ is  {\!}    invertible on $D_G$ and also {\!} on
$D_{G_0}$, because it coincides  {\!}    therein with $G_0+kI$.  Therefore
$D_G=D_{G_0}$.

Observe now that the identity $B +Q_\infty B^*Q_\infty^{-1} =-Q Q_\infty
^{-1}$ yields  {\!}    ${\rm Tr}(B) + {\rm Tr}( Q_\infty B^*Q_\infty^{-1})= -{\rm
Tr}(Q Q_\infty^{-1} )$ and hence ${\rm Tr}(Q_\infty B^* Q_\infty
^{-1})={\rm Tr}(B)=-k$.  Moreover {\!}     $G_0$ satisfies  {\!}    the hypoellipticity
condition. Indeed, if $E$ is  {\!}    an invariant subspace of
$Q_\infty^{-1}BQ_\infty$, contained in ${\rm Ker}(Q)$, the equation
$BQ_\infty +Q_\infty B^*=-Q$ easily implies  {\!}    that $B^* (E) \subset E$.
It follows  {\!}    that $E=\{0\}$, since $A$ is  {\!}    hypoelliptic.

Since $\sigma  (-Q_\infty B^*Q_\infty^{-1})=-\sigma  (B)\subset\cc^+$,
from \cite[Theorem 4.7]{Metafune} it follows  {\!}    that the spectrum of
$(G_0,D_{G_0})$ is  {\!}    the half-plane
$$
\bigl \{\lambda {\!} \in \cc: {\rm Re}\ \lambda {\!} \leq
{\rm Tr}(Q_\infty B^*Q_\infty^{-1})=-k\bigr {\!}     \}
$$
and that every complex number {\!}     $\lambda$ with ${\rm Re}\ \lambda {\!} <-k$
is  {\!}    an eigenvalue.  Since $G=G_0+kI$ and the spectra {\!} of $(A_1,D_1)$
and $(G,D_G)$ coincide, the proof is  {\!}    complete.
\end{proof}

Observe  that the eigenvalues  {\!}    associated to {\!} polynomial
eigenfunctions  {\!}    are the same for {\!}     all $p\geq 1$.  In fact, assuming
that the eigenfunctions  {\!}    are polynomials, the arguments
in Section 3 can be used also {\!} for {\!}     $p=1$ in order {\!}     to {\!} determine
the eigenvalues.  However {\!}     in $L^1_\mu$ there are nonpolynomial
eigenfunctions  {\!}    and the spectrum is  {\!}    much larger. Moreover {\!}     we have

\begin{corollary}\label{corL1}
The semigroup $\Tt$ does  {\!}    not map $L^1_\mu$ into {\!} $W^{1,1}_\mu$,
for {\!}     any $t>0$.
\end{corollary}
\begin{proof} Assume by contradiction that $T(t_0)(L^1_\mu)$ is
contained in $W^{1,1}_\mu$ for {\!}     some $t_0>0$. This  {\!}    implies  {\!}    that
$T(t)(L^1_\mu)\subset W^{1,1}_\mu$ for {\!}     every $t\geq t_0$. Proceeding
as  {\!}    in Lemma {\!} \ref{regularity}, we find that
$T(t)(L^1_\mu)\subset C^k(\re^N)\cap W^{k,1}_\mu$ for {\!}     every
$k\in \nat$, $t \geq k t_0$. Remark that  Lemma {\!} \ref{stima} holds
also {\!} if  $p=1$.  Arguing as  {\!}    in Proposition \ref{polinomio},
we infer {\!}     that all the eigenfunctions  {\!}    of $A_1$ are polynomials.
Thus, by Lemma {\!} \ref{riduzione}, we  deduce that the point
spectrum of $A_1$ is  {\!}    discrete. This  {\!}    is  {\!}    the desired  contradiction.
\end{proof}


\begin{acknowledgment}
The authors  {\!}    wish  to {\!}  thank  A. Rhandi {\!}  (Marrakech) and M. Rockner
(Bielefeld)  for {\!}     valuable   discussions. The third author
thanks  {\!}    the whole stochastics  {\!}    group at Bielefeld for {\!}     the kind hospitality.
\end{acknowledgment}


\end{article}
\end{document}